\documentclass[reqno]{amsart}

\usepackage{amsmath}
\usepackage{amsthm}
\usepackage{amsfonts}
\usepackage{amssymb}
\usepackage{enumerate}
\usepackage{graphicx}

\newcommand{\Le}{\mathcal{L}}
\DeclareMathOperator{\comp}{Comp}
\DeclareMathOperator{\incomp}{Incomp}
\DeclareMathOperator{\spread}{Spread}

\newcommand{\1}{\mathbf{1}}

\newcommand{\relmiddle}[1]{\mathrel{}\middle#1\mathrel{}}

\DeclareMathOperator{\dist}{dist}

\DeclareMathOperator{\supp}{supp}

\newcommand{\Prob}{\mathbb{P}}
\newcommand{\E}{\mathbb{E}}

\newcommand{\eps}{\varepsilon}

\def\R{\mathbb{R}}

\theoremstyle{plain}
  \newtheorem{theorem}{Theorem}[section]
  \newtheorem{conjecture}[theorem]{Conjecture}
  \newtheorem{lemma}[theorem]{Lemma}
  \newtheorem{corollary}[theorem]{Corollary}
  
  \newtheorem{proposition}[theorem]{Proposition}

\theoremstyle{definition}
  \newtheorem{definition}[theorem]{Definition}
  
  \newtheorem{assumption}[theorem]{Assumption}

\theoremstyle{remark}
  \newtheorem{remark}[theorem]{Remark}

\begin{document}
\title[On a conjecture concerning controllable random graphs]{On a conjecture of Godsil concerning controllable random graphs} 

\author{Sean O'Rourke}
\address{Department of Mathematics, University of Colorado at Boulder, Boulder, CO 80309 }
\email{sean.d.orourke@colorado.edu}

\author{Behrouz Touri}
\address{Department of Electrical, Computer, and Energy Engineering, University of Colorado at Boulder, Boulder, CO 80309}
\email{behrouz.touri@colorado.edu}

\begin{abstract}
It is conjectured by Godsil \cite{G} that the relative number of controllable graphs compared to the total number of simple graphs on $n$ vertices approaches one as $n$ tends to infinity.  We prove that this conjecture is true.  More generally, our methods show that the linear system formed from the pair $(W,b)$ is controllable for a large class of Wigner random matrices $W$ and deterministic vectors $b$.  The proof relies on recent advances in Littlewood--Offord theory developed by Rudelson and Vershynin \cite{RVnogaps, RVlo, Vsym}.  
\end{abstract}

\maketitle

\section{Introduction}

This paper is devoted to the study of controllability properties of linear systems and graphs.  Although controllability of linear systems is a well-developed subject, the topic has regained a great deal of attention in recent years due to its application in networked systems and distributed control \cite{gu2015controllability,barabasi}.  In particular, many researchers have provided conditions on controllability of specific linear systems with a given set of parameters.  These works include \cite{ bahman,AG2,chapman,G,magnus2,marzieh,marzieh2,alex,francesco,RME,T,jorge} where controllability properties---including controllability, minimal controllability, and minimum energy control---of linear and bilinear systems were studied. 

Motivated by these works, we study controllability properties of random systems and random graphs. Specifically, we consider linear systems whose parameters (as far as controllability is concerned) are random. We confirm a common belief that ``most systems are highly controllable'' even when one deals with systems of a very discrete nature.  For the purposes of this note, we define controllability in terms of Kalman's rank condition \cite{K2, K3, K, Klec}.  

\begin{definition}[Controllable]
Let $A$ be an $n \times n$ matrix, and let $b$ be a vector in $\R^n$.  We say the pair $(A,b)$ is \emph{controllable} if the $n \times n$ matrix
\begin{equation} \label{eq:Ab}
	\begin{pmatrix} b & Ab & \cdots & A^{n-1}b \end{pmatrix}
\end{equation}
has full rank (that is, rank $n$).  Here, the matrix in \eqref{eq:Ab} is the matrix with columns $b, Ab, \ldots, A^{n-1} b$.  We say the pair $(A,b)$ is \emph{uncontrollable} if it is not controllable.  
\end{definition}

Of particular importance is the case when $A$ is the adjacency matrix of a graph.  Recall that $G(n,p)$ is the Erd\"os--R\'enyi random graph on $n$ vertices with edge density $p$.  That is, $G(n,p)$ is a simple graph on $n$ vertices (which we shall label as $1, \ldots, n$) such that each edge $\{i,j\}$ is in $G(n,p)$ with probability $p$, independent of other edges.  In the special case when $p=1/2$, one can view $G(n,1/2)$ as random graph selected uniformly among all $2^{\binom{n}{2}}$ simple graphs on $n$ vertices.  We let $A_n$ be the zero-one adjacency matrix of $G(n,p)$.  

In this note, we verify the following conjecture due to Godsil \cite{G}.  

\begin{conjecture}[Godsil \cite{G}] \label{conj:main}
As $n \to \infty$, the probability that $(A_n, \1_n)$ is controllable approaches one, where $A_n$ is the adjacency matrix of $G(n,1/2)$ and $\1_n$ is the all-ones vector in $\mathbb{R}^n$.  
\end{conjecture}

\begin{remark}
Following Godsil \cite{G}, one can define a simple graph $G$ on $n$ vertices to be \emph{controllable} if the pair $(A, \1_n)$ is controllable, where $A$ is the adjacency matrix of $G$ and $\1_n$ is the all-ones vector.  In this language, Conjecture \ref{conj:main} states that the relative number of controllable graphs compared to the total number of simple graphs on $n$ vertices approaches one as $n$ tends to infinity.     
\end{remark}

We will prove the following slightly-stronger version of Conjecture \ref{conj:main}.  

\begin{theorem} \label{thm:Gn12}
Let $A_n$ be the adjacency matrix of $G(n,1/2)$.  Then, for any $\alpha > 0$, there exists $C > 0$ such that $(A_n, \1_n)$ is controllable with probability at least $1 - C n^{-\alpha}$.  
\end{theorem}

More generally, we also consider random graphs with loops.  We define the random graph $G(n,p,q)$ as follows.  Begin with the simple Erd\"os--R\'enyi random graph $G(n,p)$.  A loop is then added at each vertex with probability $q$, independent of all other choices.  Hence, $G(n,p,0)$ is the simple graph $G(n,p)$.  We prove the following result for $G(n,1/2, q)$.  

\begin{theorem} \label{thm:Gn12q}
Let $A_n$ be the adjacency matrix of $G(n,1/2, q)$ for some $0 \leq q \leq 1$.  Then, for any $\alpha > 0$, there exists $C > 0$ such that $(A_n, \1_n)$ is controllable with probability at least $1 - C n^{-\alpha}$.  
\end{theorem}

\subsection{Overview and outline}
Theorems \ref{thm:Gn12} and \ref{thm:Gn12q} are related to a large body of work on controllability and structural controllability of network control systems.  One of the primary focal points of previous works has been controllability criteria involving the topology and symmetry of the network.  We refer the interested reader to \cite{bahman,AG2,chapman,G,barabasi,magnus2,marzieh,marzieh2,RME,T} and references therein.  

Unlike many of the previous works on the subject, the approach taken in this note utilizes mostly stochastic methods.  Specifically, we will exploit advances in Littlewood--Offord theory recently developed by Rudelson and Vershynin \cite{RVnogaps, RVlo, Vsym}.  The authors' previous work \cite{OT} also used a stochastic approach to study controllability properties of random systems and was based on recent advances by Tao and Vu \cite{TVsimple} and Nguyen, Tao, and Vu \cite{NTV} concerning gaps between eigenvalues of random matrices. There it is shown that the controllability and minimal controllability of systems is a generic property, even for systems of a very discrete nature. 

Clearly, Theorem \ref{thm:Gn12} (and hence Conjecture \ref{conj:main}) follows from Theorem \ref{thm:Gn12q} (by taking $q =0$).  Theorem \ref{thm:Gn12q} will turn out to be a corollary of a more general result for Wigner random matrices.  In Section \ref{sec:wigner} we define the class of Wigner random matrices we will study.  In Section \ref{sec:results}, we state our main results for these Wigner matrices as well as a number of corollaries.  We then use these results to prove Theorem \ref{thm:Gn12q} in Section \ref{sec:proof}.  The rest of the paper is devoted to the proof of our results for Wigner matrices.  

\subsection{Notation}
We collect here a list of notation that will be used in the sequel.  For a vector $u$, $\|u\|$ denotes the Euclidean norm.  $\1_n$ is the all-ones vector in $\mathbb{R}^n$.  We let $u \cdot v = u^\mathrm{T} v$ denote the dot product of $u, v \in \mathbb{R}^n$.  

For a matrix $A$, $\|A\|$ denotes the spectral norm of $A$.  $A^\mathrm{T}$ is the transpose of $A$, and $A^\ast$ denotes the conjugate transpose of $A$.

We let $[n]$ denote the discrete interval $\{1, \ldots, n\}$.  $S^{n-1}$ is the unit sphere in $\mathbb{R}^n$.  

For a set $S$, $|S|$ denotes the cardinality of $S$, and $S^c$ is the complement of $S$.  Similarly, $\mathcal{E}^c$ is the complement of the event $\mathcal{E}$.  

Given a non-negative real number $z$, $\lfloor z \rfloor$ is the largest integer not greater than $z$ and $\lceil z \rceil$ is the smallest integer not less than $z$.  

We use asymptotic notation under the assumption that $n \to \infty$.  We write $o(1)$ to denote a term which tends to zero as $n \to \infty$.

\section{Wigner random matrices} \label{sec:wigner}

We consider the following class of random symmetric matrices originally introduced by Wigner \cite{W}.  

\begin{definition}[Wigner matrix]
Let $\xi$ and $\zeta$ be real random variables.  We say $W = (w_{ij})_{i,j=1}^n$ is an $n \times n$ \emph{Wigner matrix} with atom variables $\xi$ and $\zeta$ if $W$ is a real symmetric matrix whose entries satisfy the following:
\begin{itemize}
\item the entries $\{w_{ij} : 1 \leq i \leq j \leq n\}$ are independent random variables,
\item the upper triangular entries $\{w_{ij} : 1 \leq i < j \leq n\}$ are independent and identically distributed (iid) copies of $\xi$,
\item the diagonal entries $\{w_{ii} : 1 \leq i \leq n\}$ are iid copies of $\zeta$.
\end{itemize}
\end{definition}

\begin{remark}
One can similarly define complex Hermitian Wigner matrices where the upper-triangular entries are iid complex-valued random variables.  However, for the purposes of this note, we will only focus on real symmetric matrices.  
\end{remark}

Throughout the paper, we will consider various assumptions on the atom variables $\xi$ and $\zeta$.  The most general assumption will be the following.  

\begin{assumption} \label{assump:nondeg}
We say a real random variable $\xi$ satisfies Assumption \ref{assump:nondeg} if
\begin{equation} \label{eq:rvassump}
	\Prob( |\xi - \xi'| \leq \eps_0) \leq 1 - p_0, \quad \Prob(|\xi| > K_0) \leq p_0 / 4
\end{equation}
for some $\eps_0, p_0, K_0 > 0$, where $\xi'$ is an independent copy of $\xi$.  
\end{assumption}

The first bound in \eqref{eq:rvassump} is an anti-concentration inequality, which guarantees that $\xi$ is non-degenerate.  In particular, 
$$ \sup_{u \in \mathbb{R}} \Prob( |\xi - u| \leq \eps_0 ) \leq 1- p_0 \implies \Prob(|\xi - \xi'| \leq \eps_0) \leq 1- p_0 $$
by conditioning on $\xi'$.  The second bound in \eqref{eq:rvassump} controls the tail behavior of $\xi$.  

Recall that a random variable $\xi$ is called \emph{degenerate} if $\xi$ is equal to a constant, $c$, almost surely; otherwise, we say $\xi$ is \emph{non-degenerate}.  It turns out that all non-degenerate random variables satisfy Assumption \ref{assump:nondeg}.  

\begin{proposition} \label{prop:nondeg}
If $\xi$ is a non-degenerate random variable, then there exist constants $\eps_0, p_0, K_0 > 0$ such that $\xi$ satisfies \eqref{eq:rvassump}.  
\end{proposition}
\begin{proof}
Let $\xi$ be a non-degenerate random variable.  We first show that there exist $\eps_0, p_0 > 0$ such that $\Prob(|\xi - \xi'| \leq \eps_0) \leq 1 - p_0$, where $\xi'$ is an independent copy of $\xi$.  Suppose to the contrary.  Then, for every $\eps, p > 0$, we have $\Prob( |\xi - \xi'| \leq \eps) \geq 1 - p$.  Taking the limits $\eps \searrow 0$ and $p \searrow 0$, we conclude that $\Prob( |\xi - \xi'| = 0) = 1$.  In other words, $\xi - \xi' = 0$ almost surely, and thus the characteristic function of $\xi$ satisfies
$$ \left| \E [ e^{i 2\pi \theta \xi} ] \right|^2 = \E[ e^{i 2 \pi \theta (\xi - \xi')} ] = 1 $$
for every $\theta \in \mathbb{R}$.  This implies that $\xi$ is degenerate (see, for instance, Theorem 1.7 in Chapter 4 of \cite{Gut} or \cite[Lemma 1.5]{Plt}), a contradiction.  Hence, there exist $\eps_0, p_0 > 0$ such that the first bound in \eqref{eq:rvassump} holds.  

For the second bound, we observe that
$$ \lim_{K \to \infty} \Prob( |\xi| > K) = 0. $$
Thus, there exists $K_0 > 0$ such that $\Prob(|\xi| > K_0) \leq p_0 / 4$, as desired.
\end{proof}

Clearly, every degenerate random variable has zero variance.  We will often make use of this fact in its contrapositive form: every random variable with non-zero variance is non-degenerate. 

While all non-degenerate random variables satisfy Assumption \ref{assump:nondeg}, the importance of this assumption will be the values of $\eps_0, p_0, K_0$.  Indeed, many of the constants appearing in our main results will depend on these values.  Let us mention a few explicit examples.  If $\xi$ is a Bernoulli random variable, which takes values $\pm 1$ with probability $1/2$, then $\xi$ satisfies \eqref{eq:rvassump} with $\eps_0 = 1/2 $, $p_0 = 1/2$, and $K_0 = 1$.  On the other hand, if $\xi$ is a standard normal random variable, then $\xi$ satisfies \eqref{eq:rvassump} with $\eps_0 = 1/2$, $p_0 = 1/2$, and $K_0 = \sqrt{8}$.  

Besides Assumption \ref{assump:nondeg}, we will often assume that the random variable $\xi$ is symmetric.  Recall that $\xi$ is a \emph{symmetric} random variable if $\xi$ and $-\xi$ have the same distribution. 

Our strongest results will also require that the atom variables $\xi$ and $\zeta$ be sub-gaussian.  

\begin{definition}[Sub-gaussian]
We say the random variable $\xi$ is \emph{sub-gaussian with sub-gaussian moment $\kappa > 0$} if
$$ \Prob( |\xi| \geq t) \leq \kappa^{-1} \exp(- \kappa t^2) $$
for all $ t > 0$.  We simply say that a random variable $\xi$ is \emph{sub-gaussian} if there exists $\kappa > 0$ such that $\xi$ is sub-gaussian with sub-gaussian moment $\kappa$.  
\end{definition} 

There are many examples of sub-gaussian random variables.  Clearly, a Gaussian random variable is sub-gaussian.  Moreover, every bounded random variable is sub-gaussian.

\section{Main results for Wigner matrices} \label{sec:results}

We now state our main results for Wigner matrices.  In particular, for a large class of Wigner matrices $W$ and vectors $b$, we show that the pair $(W,b)$ is controllable with high probability.  From this, we will deduce Theorem \ref{thm:Gn12q}.  

For our most general results, it will be useful to work on the event that the spectral norm $\|W\|$ of the $n \times n$ Wigner matrix $W$ is bounded; recall that the spectral norm $\|W\|$ is the largest singular value of $W$.  To this end, we fix $M \geq 1$ and define the event
\begin{equation} \label{eq:def:BWM}
	\mathcal{B}_{W,M} := \{ \|W\| \leq M \sqrt{n} \}. 
\end{equation}
As we will see, for many Wigner matrices $W$, the event $\mathcal{B}_{W, M}$ holds with high probability.  

We now state our main results.  To do so, we consider two distinct cases: when $b$ is deterministic and when $b$ is a random vector.  

\subsection{Results for deterministic vectors}
We consider deterministic vectors $b \in \mathbb{R}^n$ which satisfy the following property.  

\begin{definition}[Delocalized] \label{def:delocalized}
Let $b = (b_1, \ldots, b_n)$ be a vector in $\mathbb{R}^n$, and assume $K \geq 1$ and $\delta \in (0,1)$ are given parameters.  We say $b$ is \emph{$(K,\delta)$-delocalized} if
at least $n - \lfloor \delta n \rfloor$ coordinates $b_k$ satisfy the following properties:
\begin{itemize}
\item $b_k = \frac{p_k}{q_k}$, where $p_k, q_k \in \mathbb{Z}$,
\item $|p_k|, |q_k| \leq K$,
\item $p_k, q_k \neq 0$.  
\end{itemize}
\end{definition}

Definition \ref{def:delocalized} ensures that all but $\lfloor \delta n \rfloor$ of the entires of $b$ are comparable, non-zero rational numbers.  In particular, $\1_n$ is $(K,\delta)$-delocalized for all $K \geq 1$ and every $\delta \in (0,1)$.   

\begin{theorem}[Main result: deterministic vectors] \label{thm:main:deterministic}
Let $\xi$ and $\zeta$ be real random variables.  Assume $\xi$ is a symmetric random variable which satisfies \eqref{eq:rvassump} for some $\eps_0, p_0, K_0 > 0$.  Let $W$ be an $n \times n$ Wigner matrix with atom variables $\xi$ and $\zeta$.  Let $M \geq 1$, and suppose the event $\mathcal{B}_{W, M}$ holds with probability at least $1/2$.  Fix $K \geq 1$ and $\alpha > 0$.  Then, there exist constants $C > 0$ and $\delta \in (0,1)$ (depending on $\eps_0, p_0, K_0, M, K$ and $\alpha$) such that the following holds.  Let $b$ be a deterministic vector in $\mathbb{R}^n$ which is $(K,\delta)$-delocalized.  Then, conditionally on $\mathcal{B}_{W,M}$, $(W,b)$ is controllable with probability at least $1 - C n^{-\alpha}$.  
\end{theorem}

The statement of Theorem \ref{thm:main:deterministic} is rather technical.  However, the result is extremely general and applies to a large class of atom variables $\xi$ and $\zeta$.  Specifically, no assumptions are made on $\zeta$, and $\xi$ is only assumed to satisfy \eqref{eq:rvassump}.  We now specialize Theorem \ref{thm:main:deterministic} in the following corollaries.  

\begin{corollary}[Finite fourth moment] \label{cor:fourth:deterministic}
Let $\xi$ be a symmetric real random variable with unit variance and finite fourth moment.  Assume $\zeta$ is a real random variable with finite variance.  Let $W$ be an $n \times n$ Wigner matrix with atom variable $\xi$ and $\zeta$.  Fix $K \geq 1$.  Then there exists a constant $\delta \in (0,1)$ such that the following holds.  Let $b$ be a vector in $\mathbb{R}^n$ which is $(K,\delta)$-delocalized.  Then $(W,b)$ is controllable with probability $1 - o(1)$.  (Here the rate of convergence implicit in $o(1)$ depends on $K$, $\xi$, and the mean and variance of $\zeta$.)  
\end{corollary}

\begin{corollary}[Sub-gaussian] \label{cor:subgaussian:deterministic}
Let $\xi$ be a symmetric, sub-gaussian real random variable with unit variance.  Assume $\zeta$ is a real sub-gaussian random variable.  Let $W$ be an $n \times n$ Wigner matrix with atom variable $\xi$ and $\zeta$.  Fix $K \geq 1$ and $\alpha > 0$.  Then there exist constants $C > 0$ and $\delta \in (0,1)$ (depending on $K, \alpha$, $\xi$, and $\zeta$) such that the following holds.  Let $b$ be a vector in $\mathbb{R}^n$ which is $(K,\delta)$-delocalized.  Then $(W,b)$ is controllable with probability at least $1 - C n^{-\alpha}$.  
\end{corollary}

Theorem \ref{thm:Gn12q} does not follow from Corollary \ref{cor:subgaussian:deterministic} directly because the off-diagonal entries of the adjacency matrix $A_n$ are not symmetric.  However, by shifting $A_n$ by a deterministic rank one matrix, we will be able to give an equivalent reformulation of Theorem \ref{thm:Gn12q} for which Corollary \ref{cor:subgaussian:deterministic} is applicable; we do so in Section \ref{sec:proof}.  

\subsection{Results for random vectors}
We now consider the case when the vector $b \in \mathbb{R}^n$ is random.  In particular, for this case, we do not need to assume that the atom variable $\xi$ is symmetric.  We will consider random vectors $b$ which satisfy the following assumption.  

\begin{assumption} \label{assump:random}
We assume $b = (b_1, \ldots, b_n)$ is a random vector in $\mathbb{R}^n$ whose coordinates are iid copies of a random variable $\psi$ which satisfies
$$ \Prob \left( |\psi - \psi'| \leq \eps_1 \right) \leq p_1, \quad \Prob( |\psi| > K_1) \leq p_1 / 4 $$
for some $\eps_1, p_1, K_1 > 0$, where $\psi'$ is an independent copy of $\psi$.   
\end{assumption}
\begin{remark}
The assumptions on $\psi$ are equivalent to the bounds in \eqref{eq:rvassump}.  
\end{remark}

\begin{theorem}[Main result: random vectors] \label{thm:main:random}
Let $\xi$ and $\zeta$ be real random variables.  Assume $\xi$ satisfies \eqref{eq:rvassump} for some $\eps_0, p_0, K_0 > 0$.  Let $W$ be an $n \times n$ Wigner matrix with atom variables $\xi$ and $\zeta$.  Let $M \geq 1$, and suppose the event $\mathcal{B}_{W, M}$ holds with probability at least $1/2$.  Let $\alpha > 0$, and assume $b$ is a random vector in $\mathbb{R}^n$, independent of $W$, which satisfies Assumption \ref{assump:random}.  Then, there exists a  constant $C > 0$ (depending on $\eps_0, p_0, K_0, M, \eps_1, p_1, K_1$, and $\alpha$) such that, conditionally on $\mathcal{B}_{W,M}$, $(W,b)$ is controllable with probability at least $1 - C n^{-\alpha}$.  
\end{theorem}

As before, we specialize Theorem \ref{thm:main:random} in the following corollaries.  

\begin{corollary}[Finite fourth moment] \label{cor:fourth:random}
Let $\xi$ be a real random variable with mean zero, unit variance, and finite fourth moment.  Assume $\zeta$ is a real random variable with finite variance.  Let $W$ be an $n \times n$ Wigner matrix with atom variable $\xi$ and $\zeta$.  Assume $b$ is a random vector in $\mathbb{R}^n$, independent of $W$, which satisfies Assumption \ref{assump:random}.  Then $(W,b)$ is controllable with probability $1 - o(1)$.  (Here the rate of convergence implicit in $o(1)$ depends on $\eps_1, p_1, K_1$, $\xi$, and the mean and variance of $\zeta$.)
\end{corollary}

\begin{corollary}[Sub-gaussian] \label{cor:subgaussian:random}
Let $\xi$ and $\zeta$ be real sub-gaussian random variables, and assume $\xi$ has mean zero and unit variance.  Let $W$ be an $n \times n$ Wigner matrix with atom variable $\xi$ and $\zeta$.  Assume $b$ is a random vector in $\mathbb{R}^n$, independent of $W$, which satisfies Assumption \ref{assump:random}.  Then, for every $\alpha > 0$, there exists $C > 0$ (depending on $\eps_1, p_1, K_1, \alpha$, $\xi$, and $\zeta$) such that $(W,b)$ is controllable with probability at least $1 - C n^{-\alpha}$.  
\end{corollary}

\section{Proof of Theorem \ref{thm:Gn12q} and the corollaries from Section \ref{sec:results}} \label{sec:proof}

This section is devoted to the proof of Theorem \ref{thm:Gn12q} as well as the various corollaries appearing in Section \ref{sec:results}.  We prove these results assuming Theorems \ref{thm:main:deterministic} and \ref{thm:main:random}.  We begin with the corollaries.  

\subsection{Proof of Corollaries \ref{cor:fourth:deterministic} and \ref{cor:fourth:random}}

We will need the following bound on the spectral norm of a Wigner random matrix whose off-diagonal entries have finite fourth moment.  
\begin{lemma}[Spectral norm: finite fourth moment] \label{lemma:norm:fourth}
Let $\xi$ be a real random variable with mean zero, unit variance, and finite fourth moment.  Assume $\zeta$ is a real random variable with finite variance.  Let $W$ be an $n \times n$ Wigner matrix with atom variable $\xi$ and $\zeta$.  Then, for every $\eps > 0$, there exists $M \geq 1$ (depending only on $\eps$, the fourth moment of $\xi$, and the mean and variance of $\zeta$) such that $\mathcal{B}_{W,M}$ holds with probability at least $1 - \eps$.  
\end{lemma}
\begin{proof}
Let $\mu$ and $\sigma^2$ be the mean and variance of $\zeta$.  We decompose the matrix $W$ as
$$ W := U + L + D, $$
where $D$ is a diagonal matrix (whose entries are iid copies of $\zeta$), $U$ is a strictly upper-triangular matrix (whose strictly upper-triangular entries are iid copies of $\xi$), and $L = U^\mathrm{T}$.  Since $\|W\| \leq \|U\| + \|L\| + \|D\|$, we have
\begin{equation} \label{eq:normsplit}
	\Prob( \|W\| \geq 3 M \sqrt{n} ) \leq 2 \Prob( \|U\| \geq M \sqrt{n} ) + \Prob( \|D \| \geq M \sqrt{n} ) 
\end{equation}
because $\|U\| = \|L\|$.  

Since $D$ is diagonal, it follows that $\|D\| = \max_{1 \leq i \leq n} |w_{ii}|$.  As the diagonal entries of $W$ are iid copies of $\zeta$, we have
\begin{align*}
	\Prob(\|D\| \geq M \sqrt{n} ) \leq n \Prob(|\zeta| \geq M \sqrt{n} ) 	\leq \frac{\E|\zeta|^2}{ M^2 } \leq 2 \frac{ \sigma^2 + |\mu|^2}{M^2}.
\end{align*}

From \cite[Theorem 2]{L}, it follows that
$$ \E \| U \| \leq C \sqrt{n}, $$
where $C > 0$ depends only on the fourth moment of $\xi$.  Here the results of \cite{L} are applicable since the entries of $U$ are independent and have mean zero.  Thus, by Markov's inequality, we conclude that
$$ \Prob( \|U \| \geq M \sqrt{n} ) \leq \frac{C}{M}. $$

Combining the bounds above with \eqref{eq:normsplit}, we obtain
$$ \Prob( \|W\| \geq 3 M \sqrt{n} ) \leq 2 \frac{C}{M} + 2 \frac{ \sigma^2 + |\mu|^2}{M^2} \leq \eps $$
by taking $M$ sufficiently large.  
\end{proof}

We now prove Corollaries \ref{cor:fourth:deterministic} and \ref{cor:fourth:random} assuming Theorems \ref{thm:main:deterministic} and \ref{thm:main:random}.

\begin{proof}[Proof of Corollary \ref{cor:fourth:deterministic}]
By Proposition \ref{prop:nondeg}, it follows that $\xi$ satisfies \eqref{eq:rvassump} for some $\eps_0, p_0, K_0 > 0$.  In addition, since $\xi$ is symmetric, $\xi$ must have mean zero.  Let $0 < \eps < 1/4$.  In view of Lemma \ref{lemma:norm:fourth}, there exists $M \geq 1$ such that $\mathcal{B}_{W,M}$ holds with probability at least $1 - \eps$.  Fix $K \geq 1$.  Since $W$ satisfies the assumptions of Theorem \ref{thm:main:deterministic}, there exists $C > 0$ and $\delta \in (0,1)$ such that, if $b$ is a $(K,\delta)$-delocalized vector in $\mathbb{R}^n$, then, conditionally on $\mathcal{B}_{W,M}$, $(W,b)$ is controllable with probability at least $1 - C n^{-1}$.  

Let $b$ be a $(K, \delta)$-delocalized vector in $\mathbb{R}^n$.  Define the event
$$ \mathcal{E} := \{ (W,b) \text{ is controllable} \}. $$
Then
\begin{align*}
	\Prob( \mathcal{E}^c ) &= \Prob \left( \mathcal{E}^c \relmiddle| \mathcal{B}_{W,M} \right) \Prob( \mathcal{B}_{W, M}) + \Prob \left( \mathcal{E}^c \relmiddle| \mathcal{B}_{W,M}^c \right) \Prob( \mathcal{B}_{W,M}^c ) \\
	&\leq \Prob \left( \mathcal{E}^c \relmiddle| \mathcal{B}_{M,W} \right) + \Prob ( \mathcal{B}_{W,M}^c ) \\
	&\leq \frac{C}{n} + \eps \\
	&\leq 2 \eps 
\end{align*}
for $n$ sufficiently large.  Since $\eps$ was arbitrary, the proof is complete.  
\end{proof}

The proof of Corollary \ref{cor:fourth:random} is nearly identical to the proof of Corollary \ref{cor:fourth:deterministic} except one applies Theorem \ref{thm:main:random} instead of Theorem \ref{thm:main:deterministic}; we omit the details.  

\subsection{Proof of Corollaries \ref{cor:subgaussian:deterministic} and \ref{cor:subgaussian:random}}

We will need the following bound on the spectral norm of a Wigner matrix with sub-gaussian entries.  

\begin{lemma}[Spectral norm: sub-gaussian] \label{lemma:norm:subgaussian}
Let $\xi$ and $\zeta$ be real sub-gaussian random variables, and assume $\xi$ has mean zero and unit variance.  Let $W$ be an $n \times n$ Wigner matrix with atom variable $\xi$ and $\zeta$.  Then there exists $C, c > 0$ and $M \geq 1$ (depending only on the sub-gaussian moments of $\xi$ and $\zeta$) such that $\mathcal{B}_{W,M}$ holds with probability at least $1 - C \exp(-c n)$.  
\end{lemma}
\begin{proof}
Let $\mu$ be the mean of $\zeta$.  It follows that $|\mu|$ can be bounded in terms of the sub-gaussian moment of $\zeta$.  Thus, by increasing the constant $C$ if necessary, it suffices to assume that $|\mu| \leq \sqrt{n}$.  

We first consider the spectral norm of the matrix $W - \mu I$, where $I$ is the identity matrix.  In particular, the entries of $W - \mu I$ are sub-gaussian random variables with mean zero.  Thus, by \cite[Lemma 5]{OT}, there exists $C, c > 0$ and $M \geq 1$ such that 
$$ \|W - \mu I \| \leq M \sqrt{n} $$
with probability at least $1 - C \exp(-c n)$.  Hence, 
$$ \|W\| \leq \|W - \mu I\| + \| \mu I\| = \| W - \mu I \| + |\mu| \leq (M + 1) \sqrt{n} $$
on the same event.  
\end{proof}

We now prove Corollaries \ref{cor:subgaussian:deterministic} and \ref{cor:subgaussian:random}.

\begin{proof}[Proof of Corollary \ref{cor:subgaussian:deterministic}]
By Proposition \ref{prop:nondeg}, it follows that $\xi$ satisfies \eqref{eq:rvassump} for some $\eps_0, p_0, K_0 > 0$.  In addition, since $\xi$ is symmetric, $\xi$ must have mean zero.  In view of Lemma \ref{lemma:norm:subgaussian}, there exist constants $C_0, c_0 > 0$ and $M \geq 1$ such that $\mathcal{B}_{W,M}$ holds with probability at least $1 - C_0 \exp(-c_0 n)$.  Fix $K \geq 1$ and $\alpha > 0$.  Since $W$ satisfies the assumptions of Theorem \ref{thm:main:deterministic}, there exists $C > 0$ and $\delta \in (0,1)$ such that, if $b$ is a $(K,\delta)$-delocalized vector in $\mathbb{R}^n$, then, conditionally on $\mathcal{B}_{W,M}$, $(W,b)$ is controllable with probability at least $1 - C n^{-\alpha}$.  

Let $b$ be a $(K, \delta)$-delocalized vector in $\mathbb{R}^n$.  Define the event
$$ \mathcal{E} := \{ (W,b) \text{ is controllable} \}. $$
Then
\begin{align*}
	\Prob( \mathcal{E}^c ) &= \Prob \left( \mathcal{E}^c \relmiddle| \mathcal{B}_{W,M} \right) \Prob( \mathcal{B}_{W, M}) + \Prob \left( \mathcal{E}^c \relmiddle| \mathcal{B}_{W,M}^c \right) \Prob( \mathcal{B}_{W,M}^c ) \\
	&\leq \Prob \left( \mathcal{E}^c \relmiddle| \mathcal{B}_{M,W} \right) + \Prob ( \mathcal{B}_{W,M}^c ) \\
	&\leq C n^{-\alpha} + C_0 \exp(-c_0 n) \\
	&\leq C' n^{-\alpha}
\end{align*}
for some constant $C' > 0$.  
\end{proof}

The proof of Corollary \ref{cor:subgaussian:random} is similar and relies on Theorem \ref{thm:main:random}; we omit the details.  

\subsection{Proof of Theorem \ref{thm:Gn12q}}

We now turn our attention to proving Theorem \ref{thm:Gn12q} using Corollary \ref{cor:subgaussian:deterministic}.  Unfortunately, one cannot apply Corollary \ref{cor:subgaussian:deterministic} directly to the adjacency matrix $A_n$ because the off-diagonal entries are not symmetric random variables.  We can overcome this obstacle by writing $A_n$ as a rank one perturbation of a Wigner matrix $W$ whose off-diagonal entries are symmetric.  We then show that controllability of $(W, \1_n)$ is equivalent to controllability of $(A_n, \1_n)$.  

We begin with a version of the Popov--Belevitch--Hautus (PBH) test for controllability, which we prove for completeness.  We refer the reader to \cite[Section 12.2]{Hlst} for further details.  

\begin{lemma}[PBH criterion] \label{lemma:PBHtest}
Let $A$ be an $n \times n$ real symmetric matrix, and assume $b \in \mathbb{R}^n$.  Then $(A,b)$ is uncontrollable if and only if there exists an eigenvector $v$ of $A$ such that $v^\mathrm{T} b = 0$.  
\end{lemma}
\begin{remark}
In its most general form, the PBH criterion also applies to non-symmetric matrices.  We state the PBH criterion in this form because we will only apply it to real symmetric matrices in this note.  
\end{remark}
\begin{proof}[Proof of Lemma \ref{lemma:PBHtest}]
Let $Q$ denote the $n \times n$ controllability matrix in \eqref{eq:Ab}.  To begin, assume $v$ is an eigenvector of $A$ with corresponding eigenvalue $\lambda$ and $v^\mathrm{T} b = 0$.  Then $v^\mathrm{T} A = \lambda v^\mathrm{T}$.  Thus, for any $k = 0, \ldots, n-1$, 
$$ v^\mathrm{T} A^k b = \lambda^k v^\mathrm{T} b = 0, $$
where we use the convention that $A^0$ is the identity matrix.  Therefore, $v^\mathrm{T} Q = 0$, and hence $(A,b)$ is uncontrollable.  

Conversely, suppose $v^\mathrm{T} b \neq 0$ for every eigenvector $v$ of $A$.  Clearly, $b \neq 0$.  In addition, this condition implies that each eigenspace of $A$ has dimension one.  Indeed, if $A$ has an eigenspace of dimension at least two, then that eigenspace would have a non-trivial intersection with the orthogonal complement of the space spanned by $b$.  This would contradict the fact that $v^\mathrm{T} b \neq 0$ for every eigenvector $v$.  

Since $A$ is real symmetric and each eigenspace has dimension one, it follows that $A$ must have $n$ distinct eigenvalues (i.e. the spectrum of $A$ is simple).  Thus, $(A,b)$ is controllable by the criterion test given in \cite[Lemma 1]{OT}.  
\end{proof}

Using the PBH criterion, we obtain the following equivalence.  

\begin{lemma} \label{lemma:PBHtest2}
Let $A$ be an $n \times n$ real symmetric matrix.  Let $b \in \mathbb{R}^n$ and $\gamma \in \mathbb{R}$.  Then $(A,b)$ is controllable if and only if $(A + \gamma b b^\mathrm{T}, b)$ is controllable.  
\end{lemma}
\begin{proof}
Applying the PBH criterion (Lemma \ref{lemma:PBHtest}) twice, we obtain
\begin{align*}
	(A,b) \text{ is uncontrollable} &\iff \exists \lambda, v \text{ such that } v \neq 0, A v = \lambda v, \text{ and } b^\mathrm{T} v = 0 \\
	&\iff \exists \lambda, v \text{ such that } v \neq 0, (A + \gamma b b^\mathrm{T})v = \lambda v, \text{ and } b^\mathrm{T} v = 0 \\
	&\iff (A + \gamma b b^\mathrm{T}, b) \text{ is uncontrollable},
\end{align*}
as required.  
\end{proof}

We now prove Theorem \ref{thm:Gn12q}.  Let $A_n$ be the adjacency matrix of $G(n,1/2,q)$.  Let $\xi$ be a Bernoulli random variable, which takes values $\pm 1$ with probability $1/2$, and let $\zeta$ be the random variable
$$ \zeta := \left\{
     \begin{array}{rl}
       1, & \text{ with probability } q, \\
       -1, & \text{ with probability } 1 - q.
     \end{array}
   \right. $$
Let $W$ be the $n \times n$ Wigner matrix with atom variables $\xi$ and $\zeta$.  Since $\1_n \1_n^\mathrm{T}$ is the all-ones matrix, it follows that $A_n$ has the same distribution as $\frac{1}{2} ( W + \1_n \1_n^\mathrm{T} )$.  Thus, it suffices to bound the probability that $ \left( \frac{1}{2}  ( W + \1_n \1_n^\mathrm{T} ), \1_n \right)$ is controllable.  By invariance of scaling, this is equivalent to the probability that $(W + \1_n \1_n^\mathrm{T}, \1_n)$ is controllable.  By Lemma \ref{lemma:PBHtest2}, this is the same as the probability that $(W,\1_n)$ is controllable.  The claim now follows from Corollary \ref{cor:subgaussian:deterministic}.  Here, Corollary \ref{cor:subgaussian:deterministic} is applicable because $\xi$ and $\zeta$ are sub-gaussian random variables and $\xi$ is symmetric with unit variance.  In addition, $\1_n$ is $(K, \delta)$-delocalized for all $K \geq 1$ and every $\delta \in (0,1)$.

\section{Reduction to a small ball probability involving eigenvectors}

It remains to verify Theorems \ref{thm:main:deterministic} and \ref{thm:main:random}.  In this section, we begin with a few reductions, which will reduce the proofs of these two theorems to a problem involving the structure of the eigenvectors of the Wigner matrix $W$.  

\subsection{Wigner matrices have simple spectrum}
Recall that the spectrum of a real symmetric matrix is real.  We say the spectrum is \emph{simple} if all eigenvalues have multiplicity one.  Recently Tao and Vu \cite{TVsimple} verified that Wigner random matrices have simple spectrum with high probability.  

\begin{theorem}[Tao-Vu] \label{thm:simple}
Let $\xi$ and $\zeta$ be real random variables.  Assume $\xi$ satisfies \eqref{eq:rvassump} for some $\eps_0, p_0, K_0 > 0$.  Let $W$ be an $n \times n$ Wigner matrix with atom variables $\xi$ and $\zeta$.  Then, for every $\alpha > 0$, there exists $C > 0$ (depending on $\eps_0, p_0, K_0$, and $\alpha$) such that the spectrum of $W$ is simple with probability at least $1 - C n^{-\alpha}$.  
\end{theorem}

Theorem \ref{thm:simple} follows almost immediately from \cite[Theorem 1.3]{TVsimple}.  Indeed, \cite[Theorem 1.3]{TVsimple} only requires that $\xi$ satisfy a weak non-degeneracy condition.  The proposition below shows that Assumption \ref{assump:nondeg} implies this non-degeneracy condition.

\begin{proposition} \label{prop:tvcheck}
Let $\xi$ be a real random variable which satisfies \eqref{eq:rvassump} for some $\eps_0, p_0, K_0 > 0$.  Then there exists $\eta > 0$ (depending on $p_0$) such that
$$ \sup_{u \in \mathbb{R}} \Prob( \xi = u ) \leq 1 - \eta. $$
\end{proposition}
\begin{proof}
Set $\eta := 1 - \sqrt{1 - p_0 / 2}$, and assume there exists $u \in \mathbb{R}$ such that $\Prob( \xi = u ) > 1 - \eta$.  Then, from the first bound in \eqref{eq:rvassump} and the independence of $\xi$ and $\xi'$, we have
$$ 1 - p_0 \geq \Prob( \xi = \xi') \geq \Prob( \xi = u, \xi' = u) = \left[ \Prob( \xi = u) \right]^2 \geq 1 - p_0 / 2, $$
a contradiction.  We conclude that $\Prob( \xi = u ) \leq 1 - \eta$ for all $u \in \mathbb{R}$.    
\end{proof}

Theorem \ref{thm:simple} now follows from \cite[Theorem 1.3]{TVsimple} and Proposition \ref{prop:tvcheck}.  

\subsection{Reduction to a small ball probability}
Using Theorem \ref{thm:simple}, we reduce the proof of Theorems \ref{thm:main:deterministic} and \ref{thm:main:random} to a problem involving the eigenvectors of $W$.  

\begin{theorem} \label{thm:reduction}
Let $\xi$ and $\zeta$ be real random variables.  Assume $\xi$ satisfies \eqref{eq:rvassump} for some $\eps_0, p_0, K_0 > 0$.  Let $W$ be an $n \times n$ Wigner matrix with atom variables $\xi$ and $\zeta$.  Let $M \geq 1$, and suppose the event $\mathcal{B}_{W, M}$ (defined in \eqref{eq:def:BWM}) holds with probability at least $1/2$.  Let $b$ be a vector in $\mathbb{R}^n$, independent of $W$.  Assume there exists $p$ such that every unit eigenvector $v$ of $W$ satisfies
\begin{equation} \label{eq:dotproduct0}
	\Prob \left( v^\mathrm{T} b =  0 \relmiddle| \mathcal{B}_{W,M} \right) \leq p. 
\end{equation}
Then, for every $\alpha > 0$, there exists $C > 0$ (depending on $\eps_0, p_0, K_0$, and $\alpha$) such that, conditionally on $\mathcal{B}_{W, M}$, $(W,b)$ is controllable with probability at least 
$$ 1 - np - C n^{-\alpha}. $$  
\end{theorem}
\begin{proof}
Let $\alpha > 0$.  By Theorem \ref{thm:simple}, there exists $C > 0$ (depending on $\eps_0, p_0, K_0$, and $\alpha$) such that the spectrum of $W$ is simple with probability at least $1 - C n^{-\alpha}$.  We will utilize the PBH test (Lemma \ref{lemma:PBHtest}) to show that $(W,b)$ is controllable.  In particular, when the spectrum of $W$ is simple, the unit eigenvectors of $W$ are determined uniquely up to sign.  Observe that the choice of sign for each eigenvector $v$ does not effect whether the dot product $v^\mathrm{T} b$ is zero or not.  Thus, when the spectrum of $W$ is simple, the PBH test only requires testing the $n$ orthonormal eigenvectors of $W$ (where the choice of sign for each vector is arbitrary).  Thus, by the union bound over the $n$ orthonormal eigenvectors, we obtain
\begin{align*}
	\Prob &\left( (W,b) \text{ is uncontrollable } \relmiddle| \mathcal{B}_{W, M} \right) \\
	&\leq \Prob \left( (W,b) \text{ is uncontrollable and the spectrum of } W \text{ is simple } \relmiddle| \mathcal{B}_{W,M} \right) \\
	&\qquad + \frac{\Prob \left( \text{spectrum of } W \text{ is not simple} \right)}{\Prob( \mathcal{B}_{W,M} )} \\
	&\leq np + 2 C n^{-\alpha}.
\end{align*} 
Here, in the last inequality, we used the assumption that $\mathcal{B}_{W,M}$ holds with probability at least $1/2$.  
\end{proof}

In view of Theorem \ref{thm:reduction}, the proof of Theorems \ref{thm:main:deterministic} and \ref{thm:main:random} reduces to computing $p$ such that \eqref{eq:dotproduct0} holds for all unit eigenvectors $v$.  In particular, Theorems \ref{thm:main:deterministic} and \ref{thm:main:random} will follow from the following two results.  

\begin{theorem} \label{thm:vectors:deterministic}
Let $\xi$ and $\zeta$ be real random variables.  Assume $\xi$ is a symmetric random variable which satisfies \eqref{eq:rvassump} for some $\eps_0, p_0, K_0 > 0$.  Let $W$ be an $n \times n$ Wigner matrix with atom variables $\xi$ and $\zeta$.  Let $M \geq 1$, and suppose the event $\mathcal{B}_{W, M}$ holds with probability at least $1/2$.  Fix $K \geq 1$ and $\alpha > 0$.  Then, there exist constants $C > 0$ and $\delta \in (0,1)$ (depending on $\eps_0, p_0, K_0, M, K$ and $\alpha$) such that the following holds.  Let $b$ be a deterministic vector in $\mathbb{R}^n$ which is $(K,\delta)$-delocalized.  Then, for any unit eigenvector $v$ of $W$, 
$$ \Prob \left( v^\mathrm{T} b = 0 \relmiddle| \mathcal{B}_{W,M} \right) \leq C n^{-\alpha}. $$
\end{theorem}

\begin{theorem} \label{thm:vectors:random}
Let $\xi$ and $\zeta$ be real random variables.  Assume $\xi$ satisfies \eqref{eq:rvassump} for some $\eps_0, p_0, K_0 > 0$.  Let $W$ be an $n \times n$ Wigner matrix with atom variables $\xi$ and $\zeta$.  Let $M \geq 1$, and suppose the event $\mathcal{B}_{W, M}$ holds with probability at least $1/2$.  Let $\alpha > 0$, and assume $b$ is a random vector in $\mathbb{R}^n$, independent of $W$, which satisfies Assumption \ref{assump:random}.  Then, there exists a  constant $C > 0$ (depending on $\eps_0, p_0, K_0, M, \eps_1, p_1, K_1$, and $\alpha$) such that, for any unit eigenvector $v$ of $W$, 
\begin{equation} \label{eq:randomvector}
	\Prob \left( v^\mathrm{T} b = 0 \relmiddle| \mathcal{B}_{W,M} \right) \leq C n^{-\alpha}. 
\end{equation}
\end{theorem}

\begin{remark}
In fact, as the proofs will show, our methods yields bounds on the small ball probability 
$$ \Prob \left( |v^\mathrm{T} b| \leq t \relmiddle| \mathcal{B}_{W,M} \right) $$
for all unit eigenvectors $v$ and every value of $t \geq 0$.  
\end{remark}

\subsection{Outline of the argument}
The rest of the paper is devoted to the proof of Theorems \ref{thm:vectors:deterministic} and \ref{thm:vectors:random}.  Let us now outline the main idea of the proof.  For simplicity, we will consider Theorem \ref{thm:vectors:random} first.  

In order to bound the probability in \eqref{eq:randomvector}, we will need to consider the random sum
\begin{equation} \label{eq:sumS}
	S := \sum_{k=1}^n v_k b_k, 
\end{equation}
where $v = (v_1, \ldots, v_n)$ is an eigenvector of $W$ and $b = (b_1, \ldots, b_n)$ is a random vector, independent of $v$, with iid entries.  As we shall see, estimating $\Prob(S = 0)$ is a type of Littlewood--Offord problem (discussed more below).  In recent years, many authors have studied Littlewood--Offord and inverse Littlewood--Offord problems; we refer the reader to \cite{NV, RVnogaps, RVsing, RVlo, TVsimple, TVlo, Vsym} and references therein.  We will apply some of this theory to estimate $\Prob(S = 0)$.  

In particular, it has been observed that if $\Prob (S = 0 ) > n^{-\alpha}$, then the vector $v$ must have a rich additive structure.  Intuitively, this follows since the sum in \eqref{eq:sumS} can only concentrate near zero when most of the coefficients $v_k$ are arithmetically well comparable.  On the other hand, if we take $v$ to be an eigenvector of a Wigner matrix, one expects the vector $v$ to look random and not have any rigid structure.  

In order to make this intuition rigorous, we proceed in two steps.
\begin{enumerate}
\item We will first show that a bound for $\Prob ( S = 0 )$ depends on the additive structure of the eigenvector $v$.  In particular, we utilize the least common denominator (LCD) concept from \cite{RVnogaps,RVsing, RVlo, Vsym} to measure the additive structure of any vector $v$.  
\item We will then show that, with high probability, any eigenvector of $W$ has very little additive structure.  Specifically, we will show that the LCD of any eigenvector is quite large.  
\end{enumerate}
Finally, we will combine the steps above to complete the proof of Theorem \ref{thm:vectors:random}.  

The proof of Theorem \ref{thm:vectors:deterministic} is similar, but we do not have the randomness of the vector $b$.  In this case, we will use the symmetry of the atom variable $\xi$ in order to introduce additional randomness.

\section{Small ball probabilities via the least common denominator}

In this section, we introduce small ball probabilities and the least common denominator (LCD) concept from \cite{RVnogaps,RVsing, RVlo, Vsym}.  At first, these concepts may appear completely unrelated to Theorems \ref{thm:vectors:deterministic} and \ref{thm:vectors:random}.  In the subsequent sections, the connection between these ideas and the theorems will become more apparent.  

We begin this section with the following definition.  

\begin{definition}[Small ball probabilities]
Let $Z$ be a random vector in $\mathbb{R}^n$.  The \emph{L\'{e}vy concentration function} of $Z$ is defined as
$$ \Le(Z,t) := \sup_{u \in \mathbb{R}^n} \Prob( \|Z - u\| \leq t ) $$
for all $t \geq 0$.  
\end{definition}

The L\'{e}vy concentration function bounds the \emph{small ball probabilities} for $Z$, which are the probabilities that $Z$ falls in a Euclidean ball of radius $t$. We begin with a few simple properties.  

\begin{lemma}[Restriction] \label{lemma:restriction}
Let $\xi_1, \ldots, \xi_n$ be independent random variables and $x_1, \ldots, x_n$ real numbers.  Then, for every subset of indices $J \subseteq [n]$ and every $t \geq 0$, we have
$$ \Le \left( \sum_{j=1}^n x_j \xi_j, t \right) \leq \Le \left( \sum_{j \in J} x_j \xi_j, t \right). $$
\end{lemma}
\begin{proof}
This bound follows by conditioning on the random variables $\xi_j$ with $j \not\in J$ and absorbing their contribution into the vector $u$ in the definition of the concentration function.  
\end{proof}

\begin{lemma}[Tensorization; Lemma 3.3 from \cite{RVnogaps}] \label{lemma:tensor}
Let $Z = (Z_1, \ldots, Z_n)$ be a random vector in $\mathbb{R}^n$ with independent coordinates.  Assume that there exists numbers $t_0, M \geq 0$ such that
$$ \Le(Z_j, t) \leq M(t + t_0) $$
for all $j$ and $t \geq 0$.  Then
$$ \Le(Z, t \sqrt{n} ) \leq [C M (t + t_0) ]^n $$
for some absolute constant $C > 0$ and all $t \geq 0$.  
\end{lemma}

\subsection{Littlewood--Offord theory and the LCD}

Littlewood--Offord theory is concerned with the small ball probabilities for sums of the form $\sum_{j=1}^n x_j \xi_j$, where $\xi_j$ are iid random variables and $x = (x_1, \ldots, x_n) \in S^{n-1}$ is a given coefficient vector.  In order to bound the small ball probabilities, it has been observed that one must take into account the additive structure of the coefficient vector $x$.  

The amount of additive structure in $x \in S^{n-1}$ is captured by the \emph{least common denominator} (LCD) of $x$.  If the coordinates $x_k = p_k/q_k$ are rational numbers, then a suitable measure of additive structure in $x$ is the least common denominator of these ratios.  That is, the smallest number $\theta > 0$ such that $\theta x \in \mathbb{Z}^n$.  We now work with an extension of this concept developed in \cite{RVnogaps,RVsing, RVlo, Vsym}. 

\begin{definition}[LCD]
Let $L \geq 1$.  We define the \emph{least common denominator} (LCD) of $x \in S^{n-1}$ as
$$ D_L(x) := \inf\left\{ \theta > 0 : \dist(\theta x, \mathbb{Z}^n) < L \sqrt{ \log_{+} (\theta / L)} \right\}, $$
where $\dist(v, T) := \inf_{u \in T} \|v - u\|$ is the distance from a vector $v \in \mathbb{R}^n$ to a set $T \subseteq \mathbb{R}^n$.  
\end{definition}

Clearly, one always has $D_L(x) \geq L$.  Another simple bound is the following.

\begin{lemma}[Simple lower bound for LCD; Proposition 7.4 from \cite{RVnogaps}] \label{lemma:LCDinfinity}
For every $x \in S^{n-1}$ and every $L \geq 1$, one has
$$ D_L(x) \geq \frac{1}{2 \|x\|_{\infty}}, $$
where $\|x\|_\infty$ is the $\ell^\infty$-norm of the vector $x$.  
\end{lemma}

\subsection{Small ball probabilities via LCD}
We now obtain a bound for the small ball probabilities of the sum $\sum x_k \xi_k$ in terms of the LCD $D_L(x)$.  The bound below is similar to \cite[Lemma 7.5]{RVnogaps} and \cite[Theorem 6.3]{Vsym}.  However, for technical reasons, we will actually need to consider the more general sum $\sum a_k x_k \xi_k$, where $x = (x_1, \ldots, x_n) \in S^{n-1}$ and $a_1, \ldots, a_n$ are rational coefficients.  Because of these rational coefficients, we cannot apply the results from \cite{RVnogaps, Vsym} directly.  

\begin{theorem}[Small ball probabilities via LCD] \label{thm:small}
Let $\xi_1, \ldots, \xi_n$ be iid copies of a real random variable $\xi$ which satisfies \eqref{eq:rvassump} for some $\eps_0, p_0, K_0 > 0$.  Let $K > 0$.  Then there exists $C > 0$ (depending only on $\eps_0, p_0, K_0$ and $K$) such that the following holds.  Let $x = (x_1, \ldots, x_n) \in S^{n-1}$ and consider the sum $S := \sum_{k=1}^n a_k x_k \xi_k$, where the coefficients $(a_k)_{k=1}^n$ satisfy
$$ a_k^{-1} \in \mathbb{Z}, \quad |a_k| \geq K^{-1} \quad \text{for } k = 1, \ldots, n. $$
Then, for every $L \geq p_0^{-1/2} K$ and $ t \geq 0$, one has
\begin{equation} \label{eq:smallbnd}
	\Le(S, t) \leq C L \left( t + \frac{1}{D_L(x)} \right). 
\end{equation}
\end{theorem}
\begin{remark}
We emphasis that the right-hand side of \eqref{eq:smallbnd} does not depend on the values of the coefficients $a_k$.  In particular, the LCD $D_L(x)$ only depends on the coefficient vector $x$.  
\end{remark}

When all of the coefficients $a_k = 1$, we obtain \cite[Theorem 6.3]{Vsym}, which we now state as a corollary.   

\begin{corollary} \label{cor:small}
Let $\xi_1, \ldots, \xi_n$ be iid copies of a real random variable $\xi$ which satisfies \eqref{eq:rvassump} for some $\eps_0, p_0, K_0 > 0$.  Then there exists $C > 0$ (depending only on $\eps_0, p_0$, and $K_0$) such that the following holds.  Let $x = (x_1, \ldots, x_n) \in S^{n-1}$ and consider the sum $S := \sum_{k=1}^n x_k \xi_k$.  Then, for every $L \geq p_0^{-1/2}$ and $ t \geq 0$, one has
$$ \Le(S, t) \leq C L \left( t + \frac{1}{D_L(x)} \right). $$
\end{corollary}

The proof of Theorem \ref{thm:small} is based on Esseen's Lemma; see, for instance, \cite[Lemma 1.16]{Plt} or \cite[Section 7.3]{TVac}. 

\begin{lemma}[Esseen; Lemma 1.16 from \cite{Plt}] \label{lemma:esseen}
Let $Y$ be a real random variable.  Then
$$ \Le(Y,1) \leq C \int_{-1}^1 |\phi_Y(\theta)| d \theta, $$
where $\phi_Y(\theta) := \E[e^{2 \pi i \theta Y}]$ is the characteristic function of $Y$, and $C > 0$ is an absolute constant.  
\end{lemma}

\begin{proof}[Proof of Theorem \ref{thm:small}]
The proof is based on the arguments given in \cite{Vsym}.  By scaling $\xi$ by $\eps_0$, we can without loss of generality assume that $\eps_0 = 1$.  Applying Esseen's Lemma (Lemma \ref{lemma:esseen}) to $S/t$, we obtain
\begin{equation} \label{eq:L1prod}
	\Le(S, t) \leq C \int_{-1}^1 \prod_{k=1}^n \left| \phi \left( \frac{\theta x_k a_k}{t} \right) \right| d \theta, 
\end{equation}
by independence of $\xi_1, \ldots, \xi_n$, where 
$$ \phi(s) := \E \left[ e^{2 \pi i s \xi} \right] $$
is the characteristic function of $\xi$.  

Let $\xi'$ be independent copy of $\xi$, and let $\bar{\xi} := \xi - \xi'$.  Clearly, $\bar{\xi}$ is a symmetry random variable.  By symmetry, 
$$ |\phi(s)|^2 = \E \cos(2 \pi s \bar{\xi}). $$
Using the bound $|x| \leq \exp(-1/2(1 - x^2))$, which is valid for all $x \in \mathbb{R}$, we obtain
\begin{equation} \label{eq:L1bnd}
	|\phi(s)| \leq \exp \left( - \frac{1}{2} \left(1 - \E \cos(2 \pi s \bar{\xi}) \right) \right). 
\end{equation}

The assumptions on $\xi$ imply that the event $\mathcal{E} := \{1 \leq |\bar{\xi}| \leq 2 K_0 \}$ holds with probability at least $p_0 / 2$.  So
\begin{align*}
	1 - \E \cos(2 \pi s \bar{\xi}) &\geq \Prob(\mathcal{E}) \E\left[ 1 - \cos(2 \pi s \bar{\xi}) \relmiddle| \mathcal{E} \right] \\
		&\geq \frac{p_0}{2}  \E \left[ \frac{4}{\pi^2} \min_{q \in \mathbb{Z}} | 2 \pi s \bar{\xi} - 2 \pi q|^2 \relmiddle| \mathcal{E} \right] \\
		&\geq 8 p_0 \E \left[ \min_{q \in \mathbb{Z}} |s \bar{\xi} - q|^2 \relmiddle| \mathcal{E} \right]. 
\end{align*}
Thus,
\begin{align*}
	1 - \E \cos(2 \pi s a_k \bar{\xi}) &\geq 8 p_0 \E \left[ \min_{q \in \mathbb{Z}} |s a_k \bar{\xi} - q|^2 \relmiddle| \mathcal{E} \right] \\
		&\geq 8 p_0 |a_k|^2 \E \left[ \min_{q \in \mathbb{Z}} |s \bar{\xi} - \frac{q}{a_k} |^2 \relmiddle| \mathcal{E} \right] \\
		& \geq 8 \frac{p_0}{K^2} \E \left[ \min_{q \in \mathbb{Z}} |s \bar{\xi} - q|^2 \relmiddle| \mathcal{E} \right]
\end{align*}
since $a_k^{-1} \in \mathbb{Z}$.  Thus, subsituting the bound above into \eqref{eq:L1bnd} yields
$$ |\phi(a_k s)| \leq \exp \left( - 4 p_0 K^{-2} \E \left[ \min_{q \in \mathbb{Z}} |s \bar{\xi} - q|^2 | \mathcal{E} \right] \right). $$

Therefore, returning to \eqref{eq:L1prod}, we conclude that
\begin{align*}
	\Le(S, t) &\leq C \int_{-1}^1 \exp \left( - 4 p_0 K^{-2} \sum_{k=1}^n \E \left[ \min_{q_k \in \mathbb{Z}} \left| \frac{ \theta x_k \bar{\xi}}{t} - q_k \right|^2 \relmiddle| \mathcal{E} \right] \right) d \theta \\
	&\leq C \E \left[ \int_{-1}^1 \exp \left( - 4 p_0 K^{-2} \dist \left( \frac{\bar{\xi}\theta}{t}x, \mathbb{Z}^n \right)^2 \right) d \theta \relmiddle| \mathcal{E} \right] 
\end{align*}
by Jensen's inequality.  Thus, by definition of the event $\mathcal{E}$, 
$$ \Le(S, t) \leq 2 C \sup_{1 \leq z \leq 2 K_0} \int_{0}^1 \exp \left( - 4 p_0 K^{-2} f_z^2(\theta) \right) d \theta, $$
where
$$ f_z(\theta) := \dist \left( \frac{z \theta}{t} x , \mathbb{Z}^n \right). $$

Suppose that $t > t_0 := \frac{2 K_0}{D_L(x)}$.  Then, for every, $1 \leq z \leq 2 K_0$ and every $\theta \in [0,1]$, we have $\frac{z \theta }{t} < D_L(x)$.  By definition, this means that
$$ f_z(\theta) \geq L \sqrt{ \log_+ \left( \frac{z \theta}{t L } \right) }. $$
Thus, 
$$ \Le(S, t) \leq 2 C \sup_{z \geq 1} \int_{0}^1 \exp \left( - 4 p_0 K^{-2} L^2 \log_{+} \left(  \frac{z \theta}{t L} \right)  \right) d \theta. $$
By the change of variables $u = \frac{z \theta}{t L}$, we find that
\begin{align*}
	\Le(S, t) &\leq 2 C t L \int_0^\infty \exp \left( - 4 p_0 K^{-2} L^2 \log_+ (u) \right) du \\
	&\leq 2 C t L \left( 1 + \int_{1}^\infty u^{-4 p_0 K^{-2} L^2} du \right).
\end{align*}
Since $p_0 K^2 L^2 \geq 1$ by assumption, the integral on the right-hand side is bounded above by an absolute constant.  Thus, 
$$ \Le(S, t) \leq C_1 t L $$
for an absolute constant $C_1 > 0$.  

Finally, suppose $t \leq t_0$.  Then, from the previous estimates applied to $2 t_0$, we obtain
$$ \Le(S, t) \leq \Le(S, 2 t_0) \leq 2 C_1 L t_0 = \frac{4 C_1 K_0 L}{D_L(x)}. $$
The proof of the theorem is complete.  
\end{proof}

Theorem \ref{thm:small} and Corollary \ref{cor:small} are most useful in the case when the coefficient vector $x$ is sufficiently unstructured i.e. when $D_L(x)$ is large.  In the situation where no information is known about the structure of $x$, the following bound can be useful.  

\begin{lemma}[Small ball probabilities: simple bound] \label{lemma:smallsimple}
Let $\xi_1, \ldots, \xi_n$ be iid copies of a real random variable $\xi$ which satisfies \eqref{eq:rvassump} for some $\eps_0, p_0, K_0 > 0$.  Assume $x = (x_1, \ldots, x_n) \in S^{n-1}$.  Then 
$$ \Le \left( \sum_{k=1}^n x_k \xi_k, c \right) \leq 1 - c', $$
where $c$ and $c'$ are positive constants that may depend only on $\eps_0, p_0$, and $K_0$.  
\end{lemma}
\begin{proof}
The proof presented here is based on the arguments given in \cite[Section 7.3]{RVnogaps}.  We consider two separate cases: when $x$ has a large coordinate and when $x$ does not.  Recall that $\|x\|_{\infty}$ is the $\ell^\infty$-norm of $x$.  Assume first that
$$ \|x\|_{\infty} \geq \frac{1}{4CL (2 + \eps_0)} =: \eta, $$
where $C$ is the constant from Corollary \ref{cor:small} and $L \geq p_0^{-1/2}$.  Choose a coordinate $k_0$ such that $|x_{k_0}| = \|x\|_{\infty}$.  Then, by Lemma \ref{lemma:restriction}, we have
$$ \Le \left( \sum_{k=1}^n x_k \xi_k, \frac{\eps_0 \eta}{2} \right) \leq \Le\left(x_{k_0} \xi_{k_0}, \frac{\eps_0 \eta}{2}\right) \leq \Le\left(\xi_{k_0}, \frac{\eps_0}{2}\right) \leq 1 - \left( 1 - \sqrt{1 - p_0/2} \right), $$
where the last inequality follows from Proposition \ref{prop:rvassumplevy} below.  

If $\|x\|_{\infty} < \eta$, then Lemma \ref{lemma:LCDinfinity} implies that
$$ D_L(x) \geq \frac{1}{2 \|x \|_{\infty}} > \frac{1}{2 \eta}. $$
So, by Corollary \ref{cor:small}, we conclude that
$$ \Le \left( \sum_{k=1}^n x_k \xi_k, \frac{\eps_0 \eta}{2} \right) \leq CL \left( \frac{\eps_0 \eta}{2} + 2 \eta \right) \leq CL \eta (\eps_0 + 2) \leq \frac{1}{4}. $$
Here the last inequality follows from the definition of $\eta$.
\end{proof}

\begin{proposition} \label{prop:rvassumplevy}
Let $\xi$ be a real random variable which satisfies \eqref{eq:rvassump} for some $\eps_0, p_0, K_0 > 0$.  Then
$$ \Le \left( \xi, \eps_0/2 \right) \leq \sqrt{1 - p_0/2}. $$
\end{proposition}
\begin{proof}
In order to reach a contradiction, assume there exists $u \in \mathbb{R}$ such that
$$ \Prob \left( |\xi - u| \leq \eps_0/2 \right) > \sqrt{1 - p_0/2}. $$
Let $\xi'$ denote an independent copy of $\xi$.  Then
\begin{align*}
	1 - p_0 \geq \Prob( |\xi - \xi'| \leq \eps_0) \geq \Prob(|\xi - u| \leq \eps_0/2, |\xi' - u| \leq \eps_0/2) > 1 - p_0/2
\end{align*}
by independence.  Since this is a contradiction, we conclude that 
\[ \sup_{u \in \mathbb{R}} \Prob(|\xi - u| \leq \eps_0/2) \leq \sqrt{1 - p_0}, \]  
as required.
\end{proof}

\section{Eigenvectors are incompressible}

Before we can show that the eigenvectors are arithmetically unstructured, we first prove a much simpler result.  In particular, we now show that the eigenvectors of a Wigner matrix are \emph{incompressible}.  An incompressible vector is one whose mass is not concentrated on a small number of coordinates.  The partition of the unit sphere into compressible and incompressible vectors has been useful in studying the invertibility of random matrices; see, for example, \cite{RVsing, RVlo, Vsym}.   

\begin{definition}[Compressible and incompressible vectors] \label{def:comp}
Let $c_0, c_1 \in (0,1)$ be two given parameters.  A vector $x \in \mathbb{R}^n$ is called \emph{sparse} if its support satisfies $|\supp(x)| \leq c_0 n$.  A vector $x \in S^{n-1}$ is called \emph{compressible} if $x$ is within Euclidean distance $c_1$ from the set of all sparse vectors.  A vector $x \in S^{n-1}$ is called \emph{incompressible} if it is not compressible.  The sets of compressible and incompressible vectors in $S^{n-1}$ will be denote by $\comp(c_0, c_1)$ and $\incomp(c_0,c_1)$ respectively. 
\end{definition}

The definition above depends on the choice of the constants $c_0, c_1$.  These constants will be chosen in Lemma \ref{lemma:complower} below.  

The sets of compressible and incompressible vectors each have their own advantages.  The set of compressible vectors has small covering numbers, which are exponential in $c_0n$ rather than $n$; see Lemma \ref{lemma:net} below for details.  On the other hand, the class of incompressible vectors has a different advantage.  Each incompressible vector $x$ has a set of coordinates of size proportional to $n$, whose magnitudes are all of order $n^{-1/2}$:

\begin{lemma}[Incompressible vectors are spread; Lemma 3.4 from \cite{RVlo}] \label{lemma:incompressible}
For every $x \in \incomp(c_0, c_1)$, one has
\begin{equation} \label{eq:incompressible}
	\frac{c_1}{\sqrt{2n}} \leq |x_k| \leq \frac{1}{\sqrt{c_0 n}} 
\end{equation}
for at least $\frac{1}{2} c_0 c_1^2 n$ coordinates $x_k$ of $x$.  
\end{lemma}

We shall now show that a Wigner matrix $W$ has no compressible unit eigenvectors.  As before we work on the event $\mathcal{B}_{W, M}$ defined in \eqref{eq:def:BWM}.  In particular, on this event, all the eigenvalues of $W$ are at most $M \sqrt{n}$ is magnitude.  Thus, we will show that\footnote{For convenience of notation we skip the identity symbol and write $W - \lambda$ for $W - \lambda I$}, with high probability, $\|(W - \lambda)x \| > 0$ for all compressible vectors $x$ and all choices of $\lambda \in \mathbb{R}$ with $|\lambda| \leq M \sqrt{n}$.  From this we can easily deduce that, conditionally on $\mathcal{B}_{W,M}$, the unit eigenvectors of $W$ are incompressible.  

\begin{lemma}[Lower bound on the set of compressible vectors] \label{lemma:complower}
Let $\xi$ and $\zeta$ be real random variables.  Assume $\xi$ satisfies \eqref{eq:rvassump} for some $\eps_0, p_0, K_0 > 0$.  Let $W$ be an $n \times n$ Wigner matrix with atom variables $\xi$ and $\zeta$.  Let $M \geq 1$, and recall the event $\mathcal{B}_{W, M}$ defined in \eqref{eq:def:BWM}.  Then there exist constants $c_0, c_1 \in (0,1)$ from Definition \ref{def:comp} and constants $C, c > 0$ such that
\begin{equation} \label{eq:complower}
	\Prob \left( \inf_{x \in \comp(c_0,c_1)} \inf_{|\lambda| \leq M \sqrt{n}} \| (W - \lambda) x \| \leq c \sqrt{n} \text{ and } \mathcal{B}_{W,M} \right) \leq C \exp(-c n). 
\end{equation}
Here the constants $C, c, c_0, c_1$ depend only on $\eps_0, p_0, K_0$, and $M$.  
\end{lemma}
\begin{remark}
Lemma \ref{lemma:complower} implies that, with high probability, the kernel of $W - \lambda$ consists of incompressible vectors:
$$ \ker(W - \lambda) \cap S^{n-1} \subseteq \incomp(c_0,c_1) \quad \text{for all} \quad \lambda \in [-M \sqrt{n}, M \sqrt{n}]. $$
\end{remark}

The rest of this section is devoted to the proof of Lemma \ref{lemma:complower}.  The argument presented here is based on the arguments given in \cite[Section 9.3]{RVnogaps}.  We shall first consider the case of a fixed compressible vector $x$ and real number $\lambda$.  We will then use a union bound argument to obtain Lemma \ref{lemma:complower}.  

To begin, we introduce the following notation.  Given an $n \times n$ matrix $A = (a_{ij})_{i,j=1}^n$ and index sets $I, J \subseteq [n]$, by $A_{I \times J}$ we denote the $|I|\times |J|$ matrix $(a_{ij})_{i \in I, j \in J}$ obtained by including entries whose rows are indexed by $I$ and whose columns are indexed by $J$.  Similarly, for a vector $x \in \mathbb{R}^n$, by $x_J$ we denote the vector in $\mathbb{R}^J$ which consists of the coefficients indexed by $J$.  In this case, the complement $J^c$ of the index set $J$ is given by $[n] \setminus J$.  

We now present a general bound for a fixed vector $x$ and real number $\lambda$.  

\begin{lemma}[Wigner matrix acting on a fixed vector] \label{lemma:fixvector}
Let $\xi$ and $\zeta$ be real random variables.  Assume $\xi$ satisfies \eqref{eq:rvassump} for some $\eps_0, p_0, K_0 > 0$.  Let $W$ be an $n \times n$ Wigner matrix with atom variables $\xi$ and $\zeta$.  Then there exist constants $C, c > 0$ (depending only on $\eps_0, p_0, K_0$) such that the following holds.  Let $\lambda \in \mathbb{R}$, and assume $x \in \mathbb{R}^n$ is a unit vector.  Then
$$ \Prob \left( \|(W - \lambda) x\| \leq c \sqrt{n} \right) \leq C \exp(-c n). $$
\end{lemma}
\begin{proof}
By increasing the constant $C$ if necessary, it suffices to assume that $n \geq 16$.  Fix a unit vector $x \in \mathbb{R}^n$.  Let $J \subseteq [n]$ be the set of indices of the $\lceil n/4 \rceil$ largest coordinates.  Since $x$ is a unit vector, it follows that $\|x_J\|^2 \geq 1/4$.  We decompose $W - \lambda$ into the sub-matrices $W_{J \times J} - \lambda$, $W_{J \times J^c}$, $W_{J^c \times J}$, and $W_{J^c \times J^c} - \lambda$.  Thus, 
\begin{align*}
	\| (W- \lambda)x\|^2 \geq \| W_{J^c \times J} x_J + (W_{J^c \times J^c} - \lambda) x_{J^c} \|^2 = \|W_{J^c \times J} x_J + a \|^2,
\end{align*}
where $a := (W_{J^c \times J^c }- \lambda) x_{J^c}$.  Observe that $W_{J^c \times J^c}$ is independent of $W_{J^c \times J}$.  Thus, we condition on $W_{J^c \times J^c}$ and now treat $a = (a_i)_{i \in J^c}$ as a constant vector.  Therefore, we have
\begin{equation} \label{eq:Wlambdafirst}
	\| (W - \lambda) x \|^2 \geq \sum_{i \in [n] \setminus J} V_i^2, 
\end{equation}
where
$$ V_i := \sum_{j \in J} w_{ij} x_j + a_i. $$
By assumption, the entries of $W_{J^c \times J}$ are iid copies of $\xi$.  Applying Lemma \ref{lemma:smallsimple}, there exist constants $c, c' > 0$, which depend only on $\eps_0, p_0$, and $K_0$, such that
$$ \Prob(|V_i| \leq c) \leq 1 - c' $$
for all $i \in [n] \setminus J$.  

Assume that $\|(W - \lambda)x\|^2 \leq \alpha c^2 n$, where $\alpha \in (0,1/8)$ is a constant to be chosen later.  In view of \eqref{eq:Wlambdafirst}, this assumption implies that
$$ \sum_{i \in [n] \setminus J} V_i^2 \leq \alpha c^2 n, $$
which in turn implies that $|V_i| \leq c$ for at least $|[n] \setminus J| - \lceil \alpha n \rceil$ of the random variables $V_i$ in the sum.  As the random variables $\{V_i : i \in [n] \setminus J \}$ are independent, we obtain
\begin{align*}
	\Prob \left( \|(W - \lambda) x \|^2 \leq \alpha c^2 n \right) &\leq \binom{ |[n] \setminus J|}{\lceil \alpha n \rceil} (1 - c')^{|[n] \setminus J| - \lceil \alpha n \rceil} \\
		&\leq \left( \frac{ |[n] \setminus J| e}{\lceil \alpha n \rceil} \right)^{\lceil \alpha n \rceil} (1 - c')^{| [n] \setminus J| - \lceil \alpha n \rceil} \\
		&\leq \left( \frac{3e}{4 \alpha} \right)^{\alpha n + 1} \exp \left( \frac{1}{2} n \log (1 - c') \right) \\
		&\leq \left(\frac{3e}{4 \alpha} \right) \exp \left( \alpha n \log \left( \frac{3e}{4 \alpha} \right) \right) \exp \left( \frac{1}{2} n \log (1 - c') \right).
\end{align*}
Here, we used the bound
$$ |[n] \setminus J| - \lceil \alpha n \rceil \geq n - \frac{n}{4} - 1 - \frac{n}{8} - 1 \geq \frac{5}{8} n - 2 \geq \frac{n}{2}, $$
which holds true for $n \geq 16$.  Choosing $\alpha$ sufficiently small so that
$$ \exp \left( \alpha n \log \left( \frac{3e}{4 \alpha} \right) \right) \leq \exp \left( - \frac{1}{4} n \log (1 - c') \right) $$
completes the proof.  
\end{proof}

Our plan is to apply Lemma \ref{lemma:fixvector} and a union bound to obtain Lemma \ref{lemma:complower}.  However, the set $\comp(c_0, c_1)$ is uncountable, and so we cannot apply the union bound directly.  In order to overcome this issue, we introduce nets as a convenient way to discretize a compact set.  

\begin{definition}[Nets]
Let $X \subseteq \mathbb{R}^d$, and let $\eps > 0$.  A subset $\mathcal{N}$ of $X$ is called an $\eps$-net of $X$ if, for every $x \in X$, there exists $y \in \mathcal{N}$ such that $\|x - y\| \leq \eps$.    
\end{definition}

The following estimate for the maximum size of an $\eps$-net of a sphere is well-known.

\begin{lemma} \label{lemma:net}
A unit sphere in $d$ dimensions admits an $\eps$-net of size at most 
$$ \left(1+\frac{2}{\eps} \right)^d.$$
\end{lemma}
\begin{proof}
Let $S$ be the unit sphere in question.  Let $\mathcal{N}$ be a maximal $\eps$-separated subset of $S$.  That is, $\|x-y\| \geq \eps$ for all distinct $x,y \in \mathcal{N}$ and no subset of $S$ containing $\mathcal{N}$ has this property.  Such a set can always be constructed by starting with an arbitrary point in $S$ and at each step selecting a point that is at least $\eps$ distance away from those already selected.  Since $S$ is compact, this procedure will terminate after a finite number of steps.  

We now claim that $\mathcal{N}$ is an $\eps$-net.  Suppose to the contrary.  Then there would exist $x \in S$ that is at least $\eps$ from all points in $\mathcal{N}$.  In other words, $\mathcal{N}\cup\{x\}$ would still be an $\eps$-separated subset of $S$.  This contradicts the maximal assumption above.  

We now proceed by a volume argument.  At each point of $\mathcal{N}$ we place a ball of radius $\eps/2$.  By the triangle inequality, it is easy to verify that all such balls are disjoint and lie in the ball of radius $1+\eps/2$ centered at the origin.  Comparing the volumes give
$$ |\mathcal{N}| \leq \frac{(1+\eps/2)^d}{(\eps/2)^d} = \left( 1 + \frac{2}{\eps} \right)^d. $$
\end{proof}

As alluded to above, the set of compressible vectors has small covering number.  That is, we can find a net for $\comp(c_0,c_1)$ which is exponential in $c_0n$ rather than $n$.  

\begin{lemma}[Net for compressible vectors] \label{lemma:compnet}
There exists a $(3c_1)$-net for the set $\comp(c_0, c_1)$ of cardinality at most
$$ \left( \frac{18}{c_0 c_1} \right)^{c_0 n}, $$
provided $n \geq 2/c_0$.    
\end{lemma}
\begin{proof}
We first construct a $c_1$-net for the set $S$ of sparse unit vectors.  Set $d := \lfloor c_0 n \rfloor$.  By Lemma \ref{lemma:net}, the unit sphere $S^{d-1}$ of $\mathbb{R}^d$ can be covered with at most $(3/c_1)^d$ Euclidean balls of radii $c_1$.  Therefore, the set $S$ can be covered with at most
$$ \binom{n}{d}\left( \frac{3}{c_1} \right)^d \leq \left( \frac{n e}{d} \right)^{c_0 n} \left( \frac{3}{c_1} \right)^{c_0 n} \leq \left( \frac{6 e}{c_1 c_0} \right)^{c_0 n} $$
Euclidean balls of radii $c_1$.  Here the last inequality uses the bound $d \geq \frac{1}{2} c_0 n$, which is a consequence of the assumption $n \geq 2/c_0$.  Let $\mathcal{N}$ be such a $c_1$-net of $S$.  We now claim that $\mathcal{N}$ is a $(3c_1)$-net of $\comp(c_0, c_1)$.  

Indeed, let $x \in \comp(c_0, c_1)$.  By definition, there exists a sparse vector $y \in \mathbb{R}^n$ such that
\begin{equation} \label{eq:sparseapprox}
	\|x - y\| \leq c_1.
\end{equation}
Among other things, \eqref{eq:sparseapprox} implies that
\begin{equation} \label{eq:sparseunit}
	|1 - \|y\|| = |\|x\| - \|y\|| \leq \|x - y \| \leq c_1. 
\end{equation}
Observe that $y / \|y\|$ is a sparse unit vector.  Thus, there exists $z \in \mathcal{N}$ such that 
\begin{equation} \label{eq:sparseapproxunit}
	\left\| z - \frac{y}{\|y\|} \right\| \leq c_1. 
\end{equation}
Hence, by \eqref{eq:sparseapprox}, \eqref{eq:sparseunit}, \eqref{eq:sparseapproxunit}, and the triangle inequality, 
\begin{align*}
	\|x - z\| &\leq \|x - y\| + \left\| y - \frac{y}{\|y\|} \right\| + \left\| \frac{y}{\|y\|} - z \right\| \\
		&\leq c_1 + \left| \|y\| - 1 \right| + c_1 \\
		&\leq 3 c_1.
\end{align*}
We conclude that $\mathcal{N}$ is a $(3c_1)$-net of $\comp(c_0, c_1)$, and the proof is complete.  
\end{proof}

We now complete the proof of Lemma \ref{lemma:complower}.  

\begin{proof}[Proof of Lemma \ref{lemma:complower}]
By Lemma \ref{lemma:fixvector}, there exist constants $C, c > 0$ such that, for any fixed unit vector $x \in \mathbb{R}^n$ and $\lambda \in \mathbb{R}$, 
\begin{equation} \label{eq:fixedvector}
	\Prob \left( \| (W - \lambda)x \| \leq 3 c \sqrt{n} \right) \leq C \exp(-3c n). 
\end{equation}
Set 
\begin{equation} \label{eq:def:c1}
	c_1 := \frac{c}{6 M}, 
\end{equation}
and take $c_0 > 0$ sufficiently small so that
\begin{equation} \label{eq:def:c0}
	c_0 \log \left( \frac{18}{c_0 c_1} \right) \leq c. 
\end{equation}
It suffices to assume that $n \geq 2/c_0$ (by increasing the constant $C$ in \eqref{eq:complower} if necessary).  

Let $\mathcal{N}_0$ be the $(3c_1)$-net of $\comp(c_0, c_1)$ given in Lemma \ref{lemma:compnet}.  In particular, 
\begin{equation} \label{eq:N0size}
	|\mathcal{N}_0| \leq \left( \frac{18}{c_0 c_1} \right)^{c_0 n}.
\end{equation}
Let $\mathcal{N}_M$ be a $(c\sqrt{n})$-net of $[-M \sqrt{n}, M \sqrt{n}]$.  $\mathcal{N}_M$ can be chosen so that 
\begin{equation} \label{eq:NMsize}
	|\mathcal{N}_M| \leq K,
\end{equation}
where $K$ is a constant depending only on $M$ and $c$.  

Assume the bad event in \eqref{eq:complower} occurs.  Then there exists $\lambda \in [-M \sqrt{n}, M \sqrt{n}]$ and $x \in \comp(c_0,c_1)$ such that
$$ \|(W - \lambda)x\| \leq c \sqrt{n} \quad \text{and} \quad \|W - \lambda\| \leq 2 M \sqrt{n}. $$
Hence, there exists $\lambda' \in \mathcal{N}_M$ and $x' \in \mathcal{N}_0$ such that
$$ \|x - x' \| \leq 3 c_1 \quad \text{and} \quad |\lambda - \lambda'| \leq c \sqrt{n}. $$
So, we obtain
\begin{align*}
	\| (W - \lambda') x' \| &\leq \| (W - \lambda) x' \| + |\lambda - \lambda'| \\
		&\leq \| (W - \lambda) x \| + 2 M \sqrt{n} \|x - x' \| + |\lambda - \lambda'| \\
		&\leq c \sqrt{n} + 6 M \sqrt{n} c_1 + c \sqrt{n} \\
		&\leq 3 c \sqrt{n}
\end{align*}
by our choice of $c_1$ in \eqref{eq:def:c1}.  Thus, it suffices to prove that
$$ \Prob \left( \inf_{x' \in \mathcal{N}_0} \inf_{\lambda' \in \mathcal{N}_M} \| (W - \lambda') x' \| \leq 3 c \sqrt{n} \right) \leq C' \exp(- c' n) $$
for some constants $C',c' > 0$.  To do so, we now apply the union bound.  Indeed, by \eqref{eq:fixedvector}, \eqref{eq:N0size}, and \eqref{eq:NMsize}, we obtain
\begin{align*}
	\Prob \left( \inf_{x' \in \mathcal{N}_0} \inf_{\lambda' \in \mathcal{N}_M} \| (W - \lambda') x' \| \leq 3 c \sqrt{n} \right) &\leq \sum_{x' \in \mathcal{N}_0} \sum_{\lambda' \in \mathcal{N}_M} \Prob \left( \| (W - \lambda') x ' \| \leq 3c \sqrt{n} \right) \\
		&\leq \left( \frac{18}{c_0 c_1} \right)^{c_0 n} K C \exp \left(-3 c n\right) \\
		&\leq K C \exp \left(-2 cn \right),
\end{align*}
where the last inequality follows from our choice of $c_0$ in \eqref{eq:def:c0}.  
\end{proof}

\section{Small ball probabilities via regularized LCD} \label{sec:regularizedLCD}

We now begin the proof of Theorems \ref{thm:vectors:deterministic} and \ref{thm:vectors:random}.  Let $\xi$ and $\zeta$ be real random variables.  Assume $\xi$ satisfies \eqref{eq:rvassump} for some $\eps_0, p_0, K_0 > 0$.  Let $W$ be an $n \times n$ Wigner matrix with atom variables $\xi$ and $\zeta$.  (Notice that no symmetry assumption is made on the atom variable $\xi$.)

In view of Lemma \ref{lemma:complower}, we fix the values of $c_0$ and $c_1$.  These values will remain fixed for the rest of the paper.  Let $x \in \incomp(c_0, c_1)$.  By Lemma \ref{lemma:incompressible}, at least $\frac{1}{2} c_0 c_1^2 n$ coordinates $x_k$ of $x$ satisfy \eqref{eq:incompressible}.  Following \cite{Vsym}, we define the constant
$$ c_{2} := \frac{1}{4} c_0 c_1^2. $$

For any vector $x \in \incomp(c_0, c_1)$, we assign a subset called $\spread(x) \subseteq [n]$ with
$$ |\spread(x)| = \lceil c_{2} n \rceil $$
such that \eqref{eq:incompressible} holds for all $k \in \spread(x)$.  At this point, we consider an arbitrary valid assignment of $\spread(x)$ to $x$.  The particular choice of assignment will be determined later.  In fact, for the proof of Theorem \ref{thm:vectors:random}, any choice will do; the proof of Theorem \ref{thm:vectors:deterministic} will require a more intelligent choice.  

We now define a regularized version of the LCD which allows us to consider only those coordinates in $\spread(x)$.  This new version of the LCD is designed to capture the amount of structure in the least structured part of the coefficients of $x$.  

\begin{definition}[Regularized LCD]
Let $\gamma \in (0, c_2)$ and $L \geq 1$.  Define the \emph{regularized LCD} of a vector $x \in \incomp(c_0, c_1)$ as
$$ \widehat{D}_L(x, \gamma) := \max \left\{ D_L (x_I / \|x_I\|) : I \subseteq \spread(x), |I| = \lceil \gamma n \rceil \right\}. $$
We let $I(x)$ denote the maximizing set $I$ in this definition.  
\end{definition}
\begin{remark}
Since the sets $I$ in this definition are subsets of $\spread(x)$, the bounds in \eqref{eq:incompressible} imply that
\begin{equation} \label{eq:xIbnd}
	\frac{\sqrt{\gamma} c_1}{\sqrt{2}} \leq \|x_I\| \leq \sqrt{\frac{ \lceil \gamma n \rceil}{c_0 n}}. 
\end{equation}
\end{remark}

\begin{lemma} \label{lemma:LCDlower}
For every $x \in \incomp(c_0, c_1)$ and every $\gamma \in (0,c_2)$ and $L \geq 1$, one has
$$ \widehat{D}_L(x, \gamma) \geq c \sqrt{\gamma n}. $$
Here $c \in (0,1)$ depends only on $c_0$ and $c_1$.  
\end{lemma}
\begin{proof}
Let $I$ be a subset as in the definition of $\widehat{D}_L(x, \gamma)$.  Consider the coordinates of $x_I / \|x_I\|$.  From \eqref{eq:incompressible} and \eqref{eq:xIbnd}, we find
$$ \sup_{i \in I} \frac{|x_i|}{\|x_I\|} \leq \frac{C}{\sqrt{\gamma n}}, $$
where $C > 0$ depends only on $c_0$ and $c_1$.  Thus, by Lemma \ref{lemma:LCDinfinity}, we conclude that
$$ D_L(x_I / \|x_I\|) \geq \frac{ \sqrt{\gamma n} }{ 2 C }. $$
The conclusion now follows from the definition of $\widehat{D}_L(x, \gamma)$.  
\end{proof}

\subsection{Small ball probabilities via regularized LCD}
We now state a version of Corollary \ref{cor:small} for the regularized LCD.  

\begin{theorem}[Small ball probabilities via regularized LCD] \label{thm:smallregularized}
Let $\xi_1, \ldots, \xi_n$ be iid copies of a real random variable $\xi$ which satisfies \eqref{eq:rvassump} for some $\eps_0, p_0, K_0 > 0$.  Then there exists $C > 0$ (depending only on $c_0, c_1, c_2, \eps_0, p_0$, and $K_0$) such that the following holds.  Let $x = (x_1, \ldots, x_n) \in \incomp(c_0, c_1)$.  Consider a subset $J \subseteq [n]$ such that $I(x) \subseteq J$.  Consider also the sum $S_J = \sum_{k \in J} x_k \xi_k$.  Then, for every $\gamma \in (0, c_2)$, $L \geq p_0^{-1/2}$, and $ t \geq 0$, one has
$$ \Le(S_J, t) \leq C L \left( \frac{t}{\sqrt{\gamma}} + \frac{1}{\widehat{D}_L(x, \gamma)} \right). $$
\end{theorem}
\begin{proof}
Let $I := I(x)$.  By Lemma \ref{lemma:restriction}, 
$$ \Le(S_J, t) \leq \Le(S_I, t). $$
Thus, by Corollary \ref{cor:small}, we obtain
$$ \Le(S_J, t) \leq \Le(S_I, t) \leq CL \left( \frac{t}{\|x_I\|} + \frac{1}{D_L(x_I/\|x_I\|)} \right) = CL \left( \frac{t}{\|x_I\|} + \frac{1}{\widehat{D}_L(x, \gamma)} \right). $$
Applying \eqref{eq:xIbnd}, we conclude that
$$ \Le(S_J, t) \leq CL \left( \frac{t\sqrt{2}}{c_1 \sqrt{\gamma}} + \frac{1}{\widehat{D}_L(x, \gamma)} \right), $$
and the proof is complete.  
\end{proof}

We will use Theorem \ref{thm:smallregularized} to verify the following bound on the small ball probabilities of $(W - \lambda)x$.  
\begin{theorem}[Small ball probabilities for $(W - \lambda) x$ via regularized LCD] \label{thm:smallregularizedW}
Let $\xi$ and $\zeta$ be real random variables.  Assume $\xi$ satisfies \eqref{eq:rvassump} for some $\eps_0, p_0, K_0 > 0$.  Let $W$ be an $n \times n$ Wigner matrix with atom variables $\xi$ and $\zeta$.  Let $x \in \incomp(c_0, c_1)$ and $\lambda \in \mathbb{R}$.  Then, for every $\gamma \in (0,c_2)$, $L \geq p_0^{-1/2}$, and $t \geq 0$, one has
$$ \Le( (W-\lambda)x, t \sqrt{n} ) \leq \left[ \frac{C L t}{\sqrt{\gamma}} + \frac{CL}{ \widehat{D}_L(x, \gamma)} \right]^{n - \lceil \gamma n \rceil}, $$
where the constant $C$ only depends on $c_0, c_1, c_2, \eps_0, p_0$, and $K_0$.  
\end{theorem}
\begin{proof}
The proof presented here is based on the arguments given in \cite[Section 6.3]{Vsym}.  Assume without loss of generality that $n \geq 2$, as the conclusion is trivial in the case when $n = 1$.  Our goal is to bound above the probability
$$ \Prob \left( \| (W - \lambda ) x - u \| \leq t \sqrt{n} \right) $$
for any arbitrary vector $u \in \mathbb{R}^n$.  

Let $I := I(x)$ be the maximizing set from the definition of $\widehat{D}_L(x, \gamma)$.  Following the proof of Lemma \ref{lemma:fixvector}, we decompose the matrix $W - \lambda$ into the sub-matrices $W_{I \times I} - \lambda$, $W_{I \times I^c}$, $W_{I^c \times I}$, and $W_{I^c \times I^c} - \lambda$; we similarly decompose the vectors $x$ and $u$.  Thus,
\begin{align*}
	\| (W - \lambda) x - u \|^2 &\geq \| W_{I^c \times I} x_I + (W_{I^c \times I^c}  - \lambda)x_{I^c} - u_{I^c} \|^2 = \| W_{I^c \times I} x_I + a \|,
\end{align*}
where
$$ a := (W_{I^c \times I^c} - \lambda) x_{I^c} - u_{I^c}. $$
Observe that $W_{I^c \times I^c}$ is independent of $W_{I^c \times I}$.  Thus, we condition on $W_{I^c \times I^c}$ and now treat $a = (a_j)_{j \in I^c}$ as a constant vector.  Therefore, we have
\begin{equation} \label{eq:Wlambdalow}
	\| (W - \lambda ) x - u \|^2 \geq \sum_{j \in [n] \setminus I} V_j^2, 
\end{equation}
where 
$$ V_j := \sum_{i \in I} w_{ji} x_i + a_j, \quad j \in I^c. $$

By assumption, the entries of $W_{I^c \times I}$ are iid copies of $\xi$.  From Theorem \ref{thm:smallregularized} (taking $J = I$), we obtain
$$ \Le(V_j, t) \leq CL \left( \frac{t}{\sqrt{\gamma}} + \frac{1}{\widehat{D}_L(x, \gamma)} \right), \quad j \in I^c. $$
Since $\{V_j : j \in I^c\}$ is an independent collection of random variables, we apply Lemma \ref{lemma:tensor} to obtain
$$ \Prob \left( \sqrt{ \sum_{j \in [n] \setminus I} V_j^2} \leq t \sqrt{ |I^c|} \right) \leq \left[ C' L \left( \frac{t}{ \sqrt{\gamma}} + \frac{1}{ \widehat{D}_L(x, \gamma)} \right) \right]^{|I^c|}, $$
where $C'$ depends only on $C$.  In view of \eqref{eq:Wlambdalow}, we conclude that
$$ \Prob \left( \| (W - \lambda) x - u \| \leq t \sqrt{ |I^c|} \right) \leq \left[ C' L \left( \frac{t}{ \sqrt{\gamma}} + \frac{1}{ \widehat{D}_L(x, \gamma)} \right) \right]^{|I^c|}. $$
The proof is now complete since $|I^c| = n - \lceil \gamma n \rceil \geq 1/4 n$.  
\end{proof}

\subsection{Additive structure}
With Theorem \ref{thm:smallregularizedW} in hand, we can now estimate the additive structure of the eigenvectors of a Wigner matrix.  

\begin{theorem}[Eigenvector structure] \label{thm:structure}
Let $\xi$ and $\zeta$ be real random variables.  Assume $\xi$ satisfies \eqref{eq:rvassump} for some $\eps_0, p_0, K_0 > 0$.  Let $W$ be an $n \times n$ Wigner matrix with atom variables $\xi$ and $\zeta$.  Let $M \geq 1$, and recall the event $\mathcal{B}_{W, M}$ defined in \eqref{eq:def:BWM}.  Let $\alpha > 0$.  Then there exist constants $C, c, c' > 0$ and $\gamma \in (0, c_2)$ (depending on $c_0, c_1, c_2, \eps_0, p_0, K_0, \alpha$, and $M$) such that, for every $p_{0}^{-1/2} \leq L \leq n^c$, one has
$$ \Prob \left( \exists \text{ eigenvector } v \text{ of } W : v \in \incomp(c_0, c_1), \widehat{D}_L(v, \gamma) \leq n^{\alpha} \text{ and } \mathcal{B}_{W,M} \right) \leq C n^{-c' n}. $$
\end{theorem}

In order to prove Theorem \ref{thm:structure}, we will need a net for the set of vectors with small LCD.  We begin with a definition.  
\begin{definition}[Sub-level sets of LCD] \label{def:sublevelLCD}
Fix $\gamma \in (0, c_2)$.  For every value $D \geq 1$, we define the set
$$ S_D := \left\{ x \in \incomp(c_0, c_1) : \widehat{D}_L(x, \gamma) \leq D \right\}. $$
\end{definition}

We will make use of the following bound for the covering number of $S_D \setminus S_{D/2}$, which follows directly from \cite[Proposition 7.4]{Vsym}.

\begin{lemma}[Covering sub-level sets of regularized LCD] \label{lemma:LCDnet}
There exist $\tilde{C}, \tilde{c} > 0$ which depend only on $c_0, c_1$, and such that the following holds.  Let $\gamma \in (\tilde{C}/n, c_2/3)$ and $L \geq 1$.  For every $D \geq 2$, the set $S_D \setminus S_{D/2}$ has a $(2\beta)$-net $\mathcal{N}$ such that
\begin{equation} \label{eq:def:betaN}
	\beta := \frac{L \sqrt{\log (3D)} }{ \sqrt{\gamma} D}, \quad |\mathcal{N}| \leq \left[ \frac{ \tilde{C} D }{ (\gamma n)^{\tilde{c}} }\right]^n D^{1/\gamma}. 
\end{equation}
\end{lemma}

The main advantage of Lemma \ref{lemma:LCDnet} is that it offers a substantially better bound than Lemma \ref{lemma:net} due to the presence of the term $(\lambda n)^{\tilde{c}}$.  Indeed, Lemma \ref{lemma:net} only offers the trivial bound $(3/\beta)^n$.  

We are now ready to prove Theorem \ref{thm:structure}.

\begin{proof}[Proof of Theorem \ref{thm:structure}]
The proof presented here is based on the arguments from \cite[Section 7.2]{Vsym}.  Let $\gamma \in (0, c_2/3)$ and $c \in (0, \alpha)$ be small parameters to be chosen later.  Without loss of generality, we shall assume $n$ is sufficiently large in terms of the parameters $c_0, c_1, c_2, \eps_0, p_0, K_0, \alpha$, and $M$.  This assumption is justified by simply increasing the constant $C$ appearing in the conclusion of the theorem.  For instance, we will assume that $n$ is sufficiently large so that $\gamma$ satisfies the assumptions of Lemma \ref{lemma:LCDnet} (i.e. $\gamma > \tilde{C}/n$).  

We first observe that it suffices to bound above the probability 
$$ \Prob \left( \inf_{x \in S_{n^\alpha} } \inf_{\lambda \in [-M\sqrt{n}, M \sqrt{n}]} \| (W - \lambda) x \| = 0 \text{ and } \mathcal{B}_{W, M} \right), $$
where $S_{n^\alpha}$ is defined in Definition \ref{def:sublevelLCD}.  We begin by bounding the probability of a simpler event.  Let $2 \leq D \leq n^{\alpha}$.  We will first obtain a bound for
\begin{equation} \label{eq:probSDD2}
	\Prob \left( \inf_{x \in S_D \setminus S_{D/2}} \inf_{\lambda \in [-M\sqrt{n}, M \sqrt{n}]} \| (W - \lambda) x \| = 0 \text{ and } \mathcal{B}_{W, M} \right).
\end{equation}
Indeed, choose a $(2\beta)$-net $\mathcal{N}$ of $S_{D} \setminus S_{D/2}$ according to Lemma \ref{lemma:LCDnet}.  By the assumptions on $L$ and $D$, we find
$$ \frac{n^{-\alpha}}{\sqrt{\gamma}} \leq | \beta| \leq \frac{n^{\alpha +1}}{\sqrt{\gamma}} $$
for $n$ sufficiently large.  Let $\mathcal{N}_M$ be a $(M \beta \sqrt{n})$-net of $[-M \sqrt{n}, M \sqrt{n}]$; $\mathcal{N}_M$ can be chosen so that
\begin{equation} \label{eq:NMcard}
	| \mathcal{N}_M | \leq \frac{K}{\beta} \leq K n^{\alpha}, 
\end{equation}
where $K$ is an absolute constant.

Assume the bad event in \eqref{eq:probSDD2} occurs.  Then there exist $x \in S_D \setminus S_{D/2}$ and $\lambda \in [-M \sqrt{n}, M \sqrt{n}]$ such that
$$ (W - \lambda)x = 0, \quad \|W - \lambda\| \leq 2M \sqrt{n}. $$
Thus, there exist $x' \in \mathcal{N}$ and $\lambda' \in \mathcal{N}_M$ such that
$$ \|x - x' \| \leq 2\beta, \quad |\lambda - \lambda'| \leq M \beta \sqrt{n}. $$
Hence, we obtain
\begin{align*}
	\| (W - \lambda') x' \| \leq \|(W - \lambda)x \| + 2 M \sqrt{n} \|x - x'\| + |\lambda - \lambda'| \leq 5 M \beta \sqrt{n}.
\end{align*}
Therefore, it will suffice to bound above the probability 
$$ \Prob \left( \inf_{x' \in \mathcal{N}} \inf_{\lambda' \in \mathcal{N}_M} \| (W - \lambda') x' \| \leq 5 M \beta \sqrt{n} \right). $$

By Theorem \ref{thm:smallregularizedW}, for any $x' \in \mathcal{N}$ and $\lambda' \in \mathbb{R}$, 
$$ \Prob \left( \| (W - \lambda')x' \| \leq 5 M \beta \sqrt{n} \right) \leq \left[ \frac{ CL \beta}{\sqrt{\gamma}} + \frac{CL}{D} \right]^{n - \lceil \gamma n \rceil}, $$
where we absorbed the factor $5M$ into the constant $C$.  Since the term $\beta/ \sqrt{\gamma}$ dominates the term $1/D$, we conclude that
$$ \Prob \left( \| (W - \lambda')x' \| \leq 5 M \beta \sqrt{n} \right) \leq \left[ \frac{C' n^{2c} \sqrt{\log (3 n^{\alpha} ) } }{ \sqrt{\gamma} D} \right]^{n - \lceil \gamma n \rceil}. $$
Thus, applying \eqref{eq:def:betaN}, \eqref{eq:NMcard}, and the union bound, we obtain
\begin{align*}
	\Prob &\left( \inf_{x' \in \mathcal{N}} \inf_{\lambda' \in \mathcal{N}_M} \| (W - \lambda')x' \| \leq 5 M \beta \sqrt{n} \right) \\
	&\qquad\qquad\leq |\mathcal{N} | |\mathcal{N}_M| \left[ \frac{C' n^{2c} \sqrt{\log (3 n^{\alpha} ) } }{ \sqrt{\gamma} D} \right]^{n - \lceil \gamma n \rceil} \\
	&\qquad\qquad\leq K n^{\alpha} \left[ \frac{ \tilde{C} C' n^{2c} \sqrt{\log (3n^\alpha)} }{ (\gamma n)^{\tilde{c}} \sqrt{\gamma} } \right]^n n^{\alpha/\gamma} \left[ \frac{n^{\alpha}}{C' n^{2c} \sqrt{\log (3n^{\alpha})} } \right]^{\lceil \gamma n \rceil}.
\end{align*}
Therefore, by taking $c$ and $\gamma$ sufficiently small and making simplifications, we obtain
$$ \Prob \left( \inf_{x' \in \mathcal{N}} \inf_{\lambda' \in \mathcal{N}_M} \| (W - \lambda')x' \| \leq 5 M \beta \sqrt{n} \right) \leq n^{-c' n} $$
for $n$ sufficiently large, where $c'$ depends only on $c_0, c_1, c_2, \eps_0, p_0, K_0, \alpha$, and $M$.  Primarily, we note that $c'$ does not depend on $D$.  Returning to \eqref{eq:probSDD2}, we conclude that
$$ \Prob \left( \inf_{x \in S_D \setminus S_{D/2}} \inf_{\lambda \in [-M\sqrt{n}, M \sqrt{n}]} \| (W - \lambda) x \| = 0 \text{ and } \mathcal{B}_{W, M} \right) \leq n^{-c'n}. $$

It remains to get rid of $S_{D/2}$ in the bound above.  Since $\beta$ decreases in $D$, we can apply the previous bound to $D/2$ instead of $D$ provided that $D/2 \geq 2$.  Hence, we have
$$ \Prob \left( \inf_{x \in S_{D/2} \setminus S_{D/4}} \inf_{\lambda \in [-M\sqrt{n}, M \sqrt{n}]} \| (W - \lambda) x \| = 0 \text{ and } \mathcal{B}_{W, M} \right) \leq n^{-c'n}. $$
We shall continue in this manner for $S_{D/4} \setminus S_{D/8}$, etc.  That is, we decompose
$$ S_{n^{\alpha}} = \bigcup_{k=0}^{k_0} S_{2^{-k}n^{\alpha}} \setminus S_{2^{-k-1}n^{\alpha}}, $$
where $k_0$ is the largest integer such that $2^{-k_0} n^{\alpha} \geq c'' \sqrt{\gamma n}$ and $c''$ is the constant from Lemma \ref{lemma:LCDlower}.  (Recall from Lemma \ref{lemma:LCDlower} that $S_D$ is empty whenever $D < c'' \sqrt{\gamma n}$.)  This implies that 
$$ k_0 \leq \log_2 \left( \frac{n^{\alpha}}{c'' \sqrt{\gamma n}} \right).  $$
Therefore, the union bound yields
$$ \Prob \left( \inf_{x \in S_{n^\alpha}} \inf_{\lambda \in [-M\sqrt{n}, M \sqrt{n}]} \| (W - \lambda) x \| = 0 \text{ and } \mathcal{B}_{W, M} \right) \leq k_0 n^{-c'n} \leq n^{-c''' n} $$
for $c'''$ chosen sufficiently small.  This completes the proof.  
\end{proof}

\section{Proof of Theorem \ref{thm:vectors:random}}

We now complete the proof of Theorem \ref{thm:vectors:random}.  The proof of Theorem \ref{thm:vectors:deterministic} is deferred until the next section.  

Recall that $c_0, c_1, c_2$ were fixed in Section \ref{sec:regularizedLCD}.  Let $b$ be a random vector in $\mathbb{R}^n$, independent of $W$, which satisfies Assumption \ref{assump:random}.  Our goal is to bound above
$$ \Prob \left( v^\mathrm{T} b = 0 \relmiddle| \mathcal{B}_{W,M} \right), $$
where $v$ is a unit eigenvector of $W$.  In fact, since $\mathcal{B}_{W,M}$ is assumed to hold with probability at least $1/2$, it suffices to bound above
$$ \Prob( v^\mathrm{T} b = 0 \text{ and } \mathcal{B}_{W,M} ). $$

Fix $\gamma$ as in Theorem \ref{thm:structure}, and take $L := \max\{p_0^{-1/2}, p_1^{-1/2} \}$.  Define $\mathcal{E}$ to be the event in which every unit eigenvector $v$ satisfies
\begin{equation} \label{eq:eigenvectorcommon}
	v \in \incomp(c_0, c_1), \quad D_L(v, \gamma) > n^{\alpha}. 
\end{equation}
By Lemma \ref{lemma:complower} and Theorem \ref{thm:structure}, it follows that
$$ \Prob(\mathcal{E}^c \cap \mathcal{B}_{W,M}) \leq C \exp(-cn). $$
Thus, for any unit eigenvector $v$, we have
\begin{align*}
	\Prob( v^\mathrm{T} b = 0 \text{ and } \mathcal{B}_{W,M} ) &\leq \Prob \left( v^\mathrm{T} b = 0 \relmiddle| \mathcal{E}^c \cap \mathcal{B}_{W,M} \right) \Prob( \mathcal{E}^c \cap \mathcal{B}_{W,M}) \\
	&\qquad\qquad + \Prob \left( v^\mathrm{T} b = 0 \relmiddle| \mathcal{E} \cap \mathcal{B}_{W,M} \right) \Prob(\mathcal{E} \cap \mathcal{B}_{W,M}) \\
	&\leq C \exp(-cn) + \Prob\left(v^\mathrm{T} b = 0 \relmiddle| \mathcal{E} \cap \mathcal{B}_{W,M} \right).
\end{align*}
Therefore, it remains to bound above 
$$ \Prob\left(v^\mathrm{T} b = 0 \relmiddle| \mathcal{E} \cap \mathcal{B}_{W,M} \right). $$
Fix a realization of $W$ in which $\mathcal{E} \cap \mathcal{B}_{W,M}$ holds.  This implies that the unit eigenvector $v$ satisfies \eqref{eq:eigenvectorcommon}.  Since a realization of $W$ (and hence $v$) is fixed, the only remaining randomness comes from the random vector $b$.  Also, since $b$ is independent of $W$, the fixed realization of $W$ does not effect $b$.  Hence, for this fixed unit vector $v$, we apply Theorem \ref{thm:smallregularized} (with $J = [n]$, $t=0$, and $x = v$) to obtain
$$ \Le(v^\mathrm{T} b, 0) \leq C' L n^{-\alpha}. $$
Since this is true for any realization of $W$ in which $\mathcal{E} \cap \mathcal{B}_{W,M}$ holds, we conclude that 
$$ \Prob\left(v^\mathrm{T} b = 0 \relmiddle| \mathcal{E} \cap \mathcal{B}_{W,M} \right) \leq C' L n^{-\alpha}, $$
and the proof is complete.

\section{Proof of Theorem \ref{thm:vectors:deterministic}} \label{sec:proof:deterministic}

This section is devoted to the proof of Theorem \ref{thm:vectors:deterministic}.  Recall that $c_0, c_1, c_2$ were fixed in Section \ref{sec:regularizedLCD}.  Define
$$ \delta := \frac{1}{8} c_0 c_1^2. $$

Recall that, for a deterministic vector $b$, which is $(K,\delta)$-delocalized, at least $n - \lfloor \delta n \rfloor$ coordinates of $b$ satisfy the properties listed in Definition \ref{def:delocalized}.  Thus, we let $Q_b \subseteq [n]$ be the coordinates of $b$ that do not satisfy the properties listed in Definition \ref{def:delocalized}.  Clearly, $|Q_b| \leq \lceil \delta n \rceil$.  

We now specify a choice of $\spread(x)$ for each $x \in \incomp(c_0, c_1)$.  Indeed, given a deterministic vector $b$, which is is $(K,\delta)$-delocalized, we take $\spread(x)$ to be a subset of $Q_b^c$ with 
$$ \spread(x) = \lceil c_2 n \rceil $$
and so that \eqref{eq:incompressible} holds for all $k \in \spread(x)$.  Observe that such a choice is always possible by our choice of $\delta$.  Additionally, all of our previous results hold for any choice of $\spread(x)$, and hence hold for this choice as well.  

The choice of $\spread(x)$ here depends on the vector $b$.  This will not be an issue because there will only be a single vector $b$, which is $(K,\delta)$-delocalized, in the proof below.  

Using our choice of $\spread(x)$, we now verify the following version of Theorem \ref{thm:smallregularized}.

\begin{theorem}[Small ball probabilities via regularized LCD] \label{thm:smalldelocalized}
Let $\xi_1, \ldots, \xi_n$ be iid copies of a real random variable $\xi$ which satisfies \eqref{eq:rvassump} for some $\eps_0, p_0, K_0 > 0$.  Let $K > 0$.  Then there exist constants $C, L_0 > 0$ (depending only on $c_0, c_1, c_2, \eps_0, p_0, K_0$ and $K$) such that the following holds.  Let $x = (x_1, \ldots, x_n) \in \incomp(c_0, c_1)$ and consider the sum $S := \sum_{k=1}^n b_k x_k \xi_k$, where $b = (b_1, \ldots, b_n) \in \mathbb{R}^n$ is $(K,\delta)$-delocalized.  Then, for every $\gamma \in (0, c_2)$, $L \geq L_0$, and $ t \geq 0$, one has
$$ \Le(S, t) \leq C L \left( \frac{t}{\sqrt{\gamma}} + \frac{1}{\widehat{D}_L(x, \gamma)} \right). $$
\end{theorem}
\begin{proof}
Let $I := I(x)$.  By Lemma \ref{lemma:restriction}, it follows that
$$ \Le(S, t) \leq \Le (S_I, t), $$
where $S_I := \sum_{k \in I} b_k x_k \xi _k$.  Since $I \subseteq \spread(x) \subseteq Q_b^c$, it follows that the coordinates of $b_I$ satisfy the properties listed in Definition \ref{def:delocalized}.  In particular, there exist constants $C', c' > 0$, which only depend on $K$, such that
$$ \left( \frac{b_i}{C'} \right)^{-1} \in \mathbb{Z}, \quad \left| \frac{b_i}{C'} \right| \geq c', \quad i \in I. $$
So by Theorem \ref{thm:small} (taking the coefficients $a_i := b_i / C'$), we obtain
\begin{align*}
	\Le(S_I, t) &\leq CL \left( \frac{t}{C' \|x_I\|} + \frac{1}{D_L(x_I/ \|x_I\|)} \right) \\
		&\leq CL \left( \frac{t \sqrt{2}}{C' c_1 \sqrt{\gamma}} + \frac{1}{ \widehat{D}_L(x, \gamma) } \right),
\end{align*}
where the last inequality follows from \eqref{eq:xIbnd} and the definition of the regularized LCD.  The proof is complete.  
\end{proof}

We will also require the following observation.

\begin{lemma} \label{lemma:symmetry}
Let $\xi$ and $\zeta$ be real random variables.  Assume $\xi$ is a symmetric random variable which satisfies \eqref{eq:rvassump} for some $\eps_0, p_0, K_0 > 0$.  Let $W$ be an $n \times n$ Wigner matrix with atom variables $\xi$ and $\zeta$.  Let $v = (v_i)_{i=1}^n$ be a unit eigenvector of $W$, and define 
\begin{equation} \label{eq:vpsi}
	v' := \left( \psi_1 v_1, \ldots, \psi_n v_n \right)^\mathrm{T}, 
\end{equation}
where $\psi_1, \ldots, \psi_n$ are iid Bernoulli random variables, which take values $\pm 1$ with equal probability, independent of $W$.  Then, for every $\alpha > 0$, there exists $C > 0$ (depending on $\eps_0, p_0, K_0$, and $\alpha$) such that
$$ \Prob( v \cdot b = 0 \text{ and } \mathcal{B}_{W,M}) \leq \Prob( v' \cdot b = 0 \text{ and } \mathcal{B}_{W,M} ) + C n^{-\alpha}. $$
\end{lemma}
\begin{proof}
Let $v = (v_i)_{i=1}^n$ be a unit eigenvector of $W$.  Consider the matrix $W' := ( \psi_i \psi_j w_{ij})_{i,j=1}^n$, and let $v'$ be the unit vector given in \eqref{eq:vpsi}.  Since $\xi$ is symmetric, it follows that $W'$ is equal in distribution to $W$.  A simple calculation reveals that 
\begin{equation} \label{eq:eigequiv}
	Wv = \lambda v \iff W' v' = \lambda v'. 
\end{equation}
In other words, the eigenvalues of $W$ and $W'$ are the same.  In particular, $W$ has simple spectrum if and only if $W'$ has simple spectrum.  

Define $\mathcal{A}_{W,M}$ to be the event where $\mathcal{B}_{W,M}$ holds and the spectrum of $W$ is simple.  We can similarly define $\mathcal{A}_{W',M}$ to be the event where $\mathcal{B}_{W',M}$ holds and the spectrum of $W'$ is simple.  In view of \eqref{eq:eigequiv}, however, it follows that $\mathcal{A}_{W,M} = \mathcal{A}_{W',M}$.  Hence, in what follows, we will write $\mathcal{A}_{W,M}$ to refer to both events.  

For a common eigenvalue $\lambda$ of $W$ and $W'$, let $P_\lambda$ denote the orthogonal projection onto the eigenspace of $W$ corresponding to $\lambda$, and let $P'_\lambda$ be the orthogonal projection onto the eigenspace of $W'$ corresponding to $\lambda$.  Since $W$ and $W'$ have the same distribution, $P_{\lambda}$ and $P'_\lambda$ have the same distribution.  Suppose $v$ is a unit eigenvector for $W$ corresponding to the eigenvalue $\lambda$.  From \eqref{eq:eigequiv}, it follows that $Wv = \lambda v$ and $W'v' = \lambda v'$.  Moreover, on the event $\mathcal{A}_{W,M}$, $P_{\lambda} = v v^\mathrm{T}$ and $P'_{\lambda} = v' {v'}^\mathrm{T}$.  Thus, we have
\begin{align*}
	\Prob \left( v \cdot b = 0 \relmiddle| \mathcal{A}_{W,M} \right)	& = \Prob \left( P_{\lambda} b = 0  \relmiddle| \mathcal{A}_{W,M} \right) \\
	&= \Prob \left( P'_{\lambda} b = 0  \relmiddle| \mathcal{A}_{W,M} \right) \\
	&=\Prob \left( {v'} \cdot b = 0 \relmiddle| \mathcal{A}_{W,M} \right).
\end{align*}
Rewriting this equality yields
\begin{equation} \label{eq:vb0awm}
	\Prob \left( v \cdot b = 0 \text{ and } \mathcal{A}_{W,M}  \right) = \Prob \left( {v'} \cdot b = 0 \text{ and } \mathcal{A}_{W,M} \right).
\end{equation}
Hence, by Theorem \ref{thm:simple} and \eqref{eq:vb0awm}, we conclude that
\begin{align*}
	\Prob \left( v \cdot b = 0 \text{ and } \mathcal{B}_{W,M} \right) &\leq \Prob \left( v \cdot b = 0 \text{ and } \mathcal{A}_{W,M} \right) + \Prob( \text{spectrum of } W \text{ is not simple} ) \\
	&\leq \Prob \left( {v'} \cdot b = 0 \text{ and } \mathcal{A}_{W,M} \right) + C n^{-\alpha} \\
	&\leq \Prob \left( {v'} \cdot b = 0 \text{ and } \mathcal{B}_{W,M} \right) + C n^{-\alpha},
\end{align*}
as desired.  
\end{proof}

We now complete the proof of Theorem \ref{thm:vectors:deterministic}.  Let $b$ be $(K,\delta)$-delocalized.  As discussed above, this fixes our assignment of $\spread(x)$ for each $x \in \incomp(c_0, c_1)$.   Our goal is to bound above
$$ \Prob \left( v^\mathrm{T} b = 0 \relmiddle| \mathcal{B}_{W,M} \right), $$
where $v = (v_i)_{i=1}^n$ is a unit eigenvector of $W$.  In fact, since $\mathcal{B}_{W,M}$ is assumed to hold with probability at least $1/2$, it suffices to bound above
$$ \Prob( v^\mathrm{T} b = 0 \text{ and } \mathcal{B}_{W,M} ). $$
By Lemma \ref{lemma:symmetry}, it suffices to bound above
$$ \Prob \left( \sum_{k=1}^n v_k \psi_k b_k = 0 \text{ and } \mathcal{B}_{W,M} \right), $$
where $\psi_1, \ldots, \psi_n$ are iid Bernoulli random variables, which take values $\pm 1$ with equal probability, independent of $W$.  

Fix $\gamma$ as in Theorem \ref{thm:structure}.  We will apply Theorem \ref{thm:smalldelocalized} to the sum $\sum_{k=1}^n v_k \psi_k b_k$ (taking $x_k = v_k$ and $\xi_k = \psi_k$).  Let $L_0$ be the constant from Theorem \ref{thm:smalldelocalized}, and set $L := \max\{p_0^{-1/2}, L_0, 2\}$.  Define $\mathcal{E}$ to be the event in which every unit eigenvector $v$ satisfies \eqref{eq:eigenvectorcommon}.  
By Lemma \ref{lemma:complower} and Theorem \ref{thm:structure}, it follows that
$$ \Prob(\mathcal{E}^c \cap \mathcal{B}_{W,M}) \leq C \exp(-cn). $$
Thus, for any unit eigenvector $v = (v_i)_{i =1}^n$, we have
\begin{align*}
	\Prob \left( \sum_{k=1}^n v_k \psi_k b_k = 0 \text{ and } \mathcal{B}_{W,M} \right) &\leq \Prob \left( \sum_{k=1}^n v_k \psi_k b_k = 0 \relmiddle| \mathcal{E}^c \cap \mathcal{B}_{W,M} \right) \Prob( \mathcal{E}^c \cap \mathcal{B}_{W,M}) \\
	&\quad + \Prob \left( \sum_{k=1}^n v_k \psi_k b_k = 0 \relmiddle| \mathcal{E} \cap \mathcal{B}_{W,M} \right) \Prob(\mathcal{E} \cap \mathcal{B}_{W,M}) \\
	&\leq C \exp(-cn) + \Prob\left( \sum_{k=1}^n v_k \psi_k b_k = 0 \relmiddle| \mathcal{E} \cap \mathcal{B}_{W,M} \right).
\end{align*}
Therefore, it remains to bound above
$$ \Prob\left(\sum_{k=1}^n v_k \psi_k b_k = 0 \relmiddle| \mathcal{E} \cap \mathcal{B}_{W,M} \right). $$
Fix a realization of $W$ in which $\mathcal{E} \cap \mathcal{B}_{W,M}$ holds.  This implies that the unit eigenvector $v$ satisfies \eqref{eq:eigenvectorcommon}.  Since a realization of $W$ (and hence $v$) is fixed, the only remaining randomness comes from the random variables $\psi_1, \ldots, \psi_n$.  Moreover, the fixed realization of $W$ does not effect $\psi_1, \ldots, \psi_n$ since these random variables are independent of $W$.  Hence, for this fixed unit vector $v$, we apply Theorem \ref{thm:smalldelocalized} (with $x_k = v_k$ and $\xi_k = \psi_k$) to obtain
$$ \Le \left(\sum_{k=1}^n v_k \psi_k b_k, 0 \right) \leq C' L n^{-\alpha}. $$
Since this is true for any realization of $W$ in which $\mathcal{E} \cap \mathcal{B}_{W,M}$ holds, we conclude that 
$$ \Prob\left(\sum_{k=1}^n v_k \psi_k b_k = 0 \relmiddle| \mathcal{E} \cap \mathcal{B}_{W,M} \right) \leq C' L n^{-\alpha}, $$
and the proof is complete.

\section{Conclusion}
In the previous sections, we proved some results on controllability of random systems with delocalized input vectors.  In particular, we confirmed a conjecture of Godsil \cite{G} concerning controllability of graphs and showed that the relative number of controllable graphs compared to the total number of simple graphs on $n$ vertices approaches one as $n$ tends to infinity.  

Along the way, we proved some controllability results for Wigner matrices with non-degenerate symmetric entries.  This involved studying the additive structure of the eigenvectors and applying Littlewood--Offord theory.  

Many research directions are left open for future investigations including the study of controllability properties of random systems with non-localized input vectors and a version of Conjecture \ref{conj:main} for dependent structures (such as Laplacian matrices).  Another direction of future research involves the study of random systems formed from Wigner matrices whose entries are not symmetric.


\begin{thebibliography}{99}

\bibitem{bahman} C. Aguilar, B. Gharesifard, \emph{Graph controllability classes for the laplacian leader-follower dynamics}, IEEE Transactions on Automatic Control, \textbf{60}(6):1611--1623, June 2015.

\bibitem{AG2} C. Aguilar, B. Gharesifard, \emph{Laplacian controllability classes for threshold graphs}, Linear Algebra and its Applications \textbf{471} (2015) 575--586.

\bibitem{chapman} A. Chapman, \emph{Semi-Autonomous Networks: Effective Control of Networked Systems through Protocols, Design, and Modeling}, Springer Theses, Springer International Publishing, 2015.

\bibitem{G} C. Godsil, \emph{Controllable Subsets in Graphs},  Annals of Combinatorics \textbf{16}(4), 733--744 (2012).  

\bibitem{gu2015controllability} S. Gu, F. Pasqualetti, M. Cieslak, Q.~K. Telesford, et~al., \emph{Controllability of structural brain networks}, Nature communications \textbf{6}, article number 8414, 2015.

\bibitem{Gut} A. Gut, \emph{Probability: A Graduate Course}, Springer Texts in Statistics, 2005.  

\bibitem{Hlst} J. P. Hespanha, \emph{Linear Systems Theory}, Princeton University Press, 2009.   

\bibitem{K2} R. E. Kalman, \emph{Contributions to the theory of optimal control}, Boletin de la Sociedad Matematica Mexicana \textbf{5}:102--119, 1960.

\bibitem{K3} R. E. Kalman, \emph{On the general theory of control systems}, Proc. 1st IFAC Congress, Moscow 1960, Vol. 1 (1961), pp. 481--492, Butterworth, London.

\bibitem{K} R. E. Kalman, Y. C. Ho, K. S. Narendra, \emph{Controllability of linear dynamical systems}, Contributions to differential equations, \textbf{1}(2):189--213, 1962. 

\bibitem{Klec} R. E. Kalman, \emph{Lectures on controllability and observability}, C.I.M.E. Summer Schools, Cremonese, Rome, pp. 1--151, 1969.    

\bibitem{L} R Latala, \emph{Some estimates of norms of random matrices} Proc. Amer. Math. Soc. \textbf{133} (2005), 1273--1282.

\bibitem{barabasi} Y. Liu, J. Slotine, A. Barab{\'a}si, \emph{Controllability of complex networks}, Nature \textbf{473}, 167--173, 2011.

\bibitem{magnus2} S. Martini, M. Egerstedt, A. Bicchi, \emph{Controllability analysis of multi-agent systems using relaxed equitable partitions}, International Journal of Systems, Control and Communications \textbf{2}(1-3):100--121, 2010.

\bibitem{marzieh} M. Nabi-Abdolyousefi, \emph{Controllability, Identification, and Randomness in Distributed Systems}, Springer Theses, Springer International Publishing, 2014.

\bibitem{marzieh2} M. Nabi-Abdolyousefi, M. Mesbahi, \emph{On the controllability properties of circulant networks}, IEEE Transactions on Automatic Control, \textbf{58}(12):3179--3184, 2013.

\bibitem{NV} H. Nguyen, Van Vu, \emph{Small ball probability, inverse theorems, and applications}, in Erd\H{o}s Centennial, pages 409--463, Springer (2013).  

\bibitem{NTV} H. Nguyen, T. Tao, and V. Vu, \emph{Random matrices: tail bounds for gaps between eigenvalues}, preprint available at {\tt arXiv:1504.00396}.  

\bibitem{alex} A. Olshevsky, \emph{Minimal controllability problems}, IEEE Transactions on Control of Network Systems, \textbf{1}(3):249--258, Sept 2014.

\bibitem{OT} S. O'Rourke, B. Touri, \emph{Controllability of random systems: Universality and minimal controllability}, available at {\tt arXiv:1506.03125}.

\bibitem{francesco} F. Pasqualetti, S. Zampieri, F. Bullo, \emph{Controllability metrics, limitations and algorithms for complex networks}, IEEE Transactions on Control of Network Systems, \textbf{1}(1):40--52, 2014.  

\bibitem{Plt} V. V. Petrov, \emph{Limit Theorems of Probability Theory: Sequences of Independent Random Variables}, Oxford Studies in Probability, 1995.

\bibitem{RME} A. Rahmani, M. Ji, M. Mesbahi, M. Egerstedt, \emph{Controllability of multi-agent systems from a graph-theoretic perspective}, SIAM Journal on Control and Optimization, Vol. 48, No. 1 (2009), pp. 162--186.

\bibitem{RVnogaps} M. Rudelson, R. Vershynin, \emph{No-gaps delocalization for general random matrices}, available at {\tt arXiv:1506.04012}. 

\bibitem{RVsing} M. Rudelson, R. Vershynin, \emph{Smallest singular value of a random rectangular matrix}, Communications on Pure and Applied Mathematics \textbf{62} (2009), 1707--1739.

\bibitem{RVlo} M. Rudelson, R. Vershynin, \emph{The Littlewood-Offord Problem and invertibility of random matrices}, Advances in Mathematics \textbf{218} (2008), 600--633.

\bibitem{T} H. G. Tanner, \emph{On the controllability of nearest neighbor interconnections}, 43rd IEEE Conference on Decision and Control, Vol. 3 (2004), 2467--2472.

\bibitem{TVsimple} T. Tao, V. Vu, \emph{Random matrices have simple spectrum}, available at {\tt arXiv:1412.1438}.

\bibitem{TVlo} T. Tao, V. Vu, \emph{Inverse Littlewood--Offord theorems and the condition number of random discrete matrices}, Annals of Mathematics, pages 595--632 (2009).

\bibitem{TVac} T. Tao, V. Vu, \emph{Additive combinatorics}, Cambridge Studies in Advanced Mathematics \textbf{105},  Cambridge University Press, Cambridge, 2006.  

\bibitem{Vsym} R. Vershynin, \emph{Invertibility of symmetric random matrices}, Random Structures and Algorithms \textbf{44} (2014), 135--182.

\bibitem{W} E. P. Wigner, \emph{On the distributions of the roots of certain symmetric matrices}, Ann. Math. \textbf{67}, (1958) 325--327.

\bibitem{jorge} Y. Zhao, J. Cort\'es, \emph{Gramian-based reachability metrics for bilinear networks}, available at {\tt arXiv:1509.02877}.


\end{thebibliography}
\end{document}